\documentclass[11pt]{amsart}

\usepackage[english]{babel}

\usepackage[T1]{fontenc}
\usepackage[utf8]{inputenc}
\usepackage[a4paper]{geometry}
\usepackage{textcomp}
\usepackage{mathtools}
\usepackage{amsmath}
\usepackage{amsthm}
\usepackage{amssymb}
\PassOptionsToPackage{normalem}{ulem}
\usepackage{ulem}

\usepackage[sc]{mathpazo}

\usepackage[svgnames]{xcolor}

\renewcommand{\ge}{\geqslant}

\makeatletter

\theoremstyle{plain}
\newtheorem{thm}{\protect\theoremname}
  \theoremstyle{definition}
  
  \theoremstyle{remark}
  \newtheorem{rem}[thm]{\protect\remarkname}
  \theoremstyle{plain}
  \newtheorem{prop}[thm]{\protect\propositionname}
  \theoremstyle{plain}
  
  \theoremstyle{plain}
  \newtheorem{lem}[thm]{\protect\lemmaname}
  \theoremstyle{remark}

\numberwithin{equation}{section}
\numberwithin{thm}{section}

\usepackage{babel}
\providecommand{\claimname}{Claim}
  \providecommand{\corollaryname}{Corollary}
  \providecommand{\lemmaname}{Lemma}
  \providecommand{\propositionname}{Proposition}
  \providecommand{\remarkname}{Remark}
  \providecommand{\theoremname}{Theorem}
\providecommand{\theoremname}{Theorem}

\makeatother

\allowdisplaybreaks

\makeatletter
\@namedef{subjclassname@2020}{%
  \textup{2020} Mathematics Subject Classification}
\makeatother


\begin{document}

\parindent=5mm

\begin{abstract}
In this article, we prove that a sum of solitons and breathers of
the modified Korteweg-de Vries equation (mKdV) is orbitally stable. The orbital stability is shown in
$H^{2}$. More precisely, we will show that if a solution of (mKdV)
is close enough to a sum of solitons and breathers with distinct velocities
at $t=0$ in the $H^{2}$ sense, then it stays close to this sum of
solitons and breathers for any time $t\geq0$ in the $H^2$ sense, up to space translations for solitons or
space and phase translations for breathers, provided the condition that the considered solitons and breathers
are sufficiently decoupled from each other and that  the velocities of the considered breathers are all positive, except possibly one. 
The constants that appear in this stability result do not depend on translation parameters.

From this, we deduce the orbital stability of any multi-breather of
(mKdV), provided the condition that the velocities of the considered breathers are all positive, except possibly one (the condition about the decoupling of the considered solitons and breathers between each other is not required in this setting). 
The constants that appear in this stability result depend on translation parameters of the considered solitons and breathers.
\end{abstract}

\title{Orbital stability of a sum of solitons and breathers of the modified Korteweg-de Vries equation}

\author{Alexander Semenov}

\maketitle

\section{Introduction}

\subsection{Setting of the problem}

We consider the modified Korteweg-de Vries equation: 
\begin{equation} \tag{mKdV}
\label{2mKdV}
\begin{cases}
\begin{array}{c}
u_{t}+(u_{xx}+u^{3})_{x}=0,\quad(t,x)\in\mathbb{R}\times\mathbb{R},\\
u(0,x)=u_{0}(x),\quad x\in\mathbb{R},
\end{array}\end{cases}
\end{equation}

for $u_{0}\in H^{2}(\mathbb{R})$.

\eqref{2mKdV} appears as a good approximation of some physical problems as
ferromagnetic vortices \cite{key-47}, fluid mechanics \cite{key-48},
electrodynamics \cite{key-43}, plasma physics \cite{key-40,key-42},
etc.

The Cauchy problem for \eqref{2mKdV} is locally well-posed in $H^{s}$ for
$s>-\frac{1}{2}$ \cite{HGKV}. For $s>\frac{1}{4}$, the Cauchy
problem is globally well-posed \cite{key-39}. In this paper, we will
use only basic results about the Cauchy problem: the fact that it
is globally well-posed in $H^{1}$ or $H^{2}$.

Note that the set of solutions of \eqref{2mKdV} is stable under space or
time translations or under reflexions with respect to the $x$-axis.

We have following conservation laws for a solution $u(t)$ of \eqref{2mKdV}:
\begin{equation}
M[u](t):=\frac{1}{2}\int u^{2}(t,x)dx,
\end{equation}
\begin{equation}
E[u](t):=\frac{1}{2}\int u_{x}^{2}(t,x)dx-\frac{1}{4}\int u^{4}(t,x)dx,
\end{equation}
\begin{equation}
F[u](t):=\frac{1}{2}\int u_{xx}^{2}(t,x)dx-\frac{5}{2}\int u^{2}(t,x)u_{x}^{2}(t,x)dx+\frac{1}{4}\int u^{6}(t,x)dx.
\end{equation}

Note that \eqref{2mKdV} has actually infinitely many conservation laws, because
it is integrable, like the original Korteweg-de Vries equation (KdV), which has quadratic nonlinearity \cite{key-27,key-28}. It is also a
special case of the generalized Korteweg-de Vries equation (gKdV) \cite{key-2}.

(gKdV) belongs to the family of nonlinear (focusing) dispersive equations. Other examples of equations belonging to this family are the nonlinear Schrödinger equation (NLS) \cite{key-3,key-6,key-22} and the nonlinear Klein-Gordon equation (KG) \cite{key-11,key-16}. They share a common property: they all admit special solutions called solitons, 
a bump that translates with a constant velocity without deformation. However, \eqref{2mKdV} enjoys a specific feature: it admits another
class of special solutions called breathers, which we will describe below.
We will consider here both specific solutions of \eqref{2mKdV} together: solitons and
breathers.

For $c>0$, $\kappa\in\{-1,1\}$ and $x_{0}\in\mathbb{R}$, a \emph{soliton}
parametrized by $c$, $\kappa$ and $x_{0}$ is a solution of \eqref{2mKdV}
that corresponds to a bump of constant shape that translates with a constant velocity $c$ without deformation, that is initially centered in $x_0$ and that has for sign $\kappa$. In other words, it is a solution of the form:
\begin{align}
R_{c,\kappa}(t,x;x_{0}):=\kappa Q_c(x-ct-x_0),
\end{align}
where $Q_c$ is the profile function that depends only on one variable and that is positive and pair. $Q_c\in H^1(\mathbb{R})$ should solve the following elliptic equation:
\begin{align}
\label{2GS}
Q_c''-cQ_c+Q_c^3=0.
\end{align}

It is possible to show that \eqref{2GS} has a unique solution in $H^1$, up to translations and reflexion with respect to the $x$-axis. One can show that $Q_c$ has the following expression:
\begin{align}
Q_c(x):=\left(\frac{2c}{\cosh^2(\sqrt{c}x)}\right)^\frac{1}{2}.
\end{align}

We denote $Q:=Q_1$ the basic ground state.

When $\kappa=-1$, the soliton is sometimes called an \emph{antisoliton}.

The soliton $R_{c,\kappa}(x_{0})$ travels with velocity $c$ to the right, the position of its center at a time $t$ is $x_{0}+ct$. It is exponentially localized,
depending on $c$ and $x_{0}+ct$ (the position of the bound depends
on $x_{0}+ct$, the amplitude and the exponential decay rate depend
on $c$): 
\begin{equation}
\left|R_{c,\kappa}(t,x;x_{0})\right|\leq\sqrt{2c}\exp\left(-\sqrt{c}\left|x-x_{0}-ct\right|\right).
\end{equation}
Analoguous bounds are valid for any derivative (with the same decay rate but different amplitude). This motivates the terminology ``shape parameter'' for $c$.

Solitons, in particular their stability, have been extensively studied: regarding orbital stability (in $H^1$), we refer to
Cazenave, Lions \cite{key-10} and Weinstein \cite{key-6,key-22} for (NLS) and Weinstein \cite{key-6}, Bona-Souganidis-Strauss \cite{Bona} and Martel-Merle \cite{MM} for \eqref{2mKdV}, see also Grillakis-Shatah-Strauss \cite{key-13} for a result in an abstract setting. Asymptotic stability (in $H^1$) of \eqref{2mKdV} solitons was shown by Martel-Merle \cite{key-20,key-21,key-23} and refined by Germain-Pusateri-Rousset \cite{key-29}.

For $\alpha,\beta>0$ and $x_{1},x_{2}\in\mathbb{R}$, a \emph{breather}
parametrized by $\alpha,\beta,x_{1},x_{2}$ is a solution of \eqref{2mKdV}
that has the following expression: 
\begin{equation}
B_{\alpha,\beta}(t,x;x_{1},x_{2}):=2\sqrt{2}\partial_{x}\left[\arctan\left(\frac{\beta}{\alpha}\frac{\sin(\alpha y_{1})}{\cosh(\beta y_{2})}\right)\right],
\end{equation}
where $y_{1}:=x+\delta t+x_{1}$, $y_{2}:=x+\gamma t+x_{2}$, $\delta:=\alpha^{2}-3\beta^{2}$
and $\gamma:=3\alpha^{2}-\beta^{2}$.

The breather $B_{\alpha,\beta}(x_{1},x_{2})$ travels with velocity $-\gamma$; the position of its center at a time $t$ is $-x_{2}-\gamma t$.
It is exponentially localized, depending on $\alpha$,
$\beta$ and $-x_{2}-\gamma t$ (the position of the bound depends
on $-x_{2}-\gamma t$, the coefficient depends on $\alpha$ and $\beta$,
and the exponential decay rate depends on $\beta$): 
\begin{equation}
\left|B_{\alpha,\beta}(t,x;x_{1},x_{2})\right|\leq C(\alpha,\beta)\exp\left(-\beta\left|x+x_{2}+\gamma t\right|\right),
\end{equation}
Analogous bounds are valid for any derivative (with the same decay rate but different amplitude). This motivates the terminology ``shape'' and ``frequency'' parameters for $\beta$ and $\alpha$, respectively.

One doesn't talk of ``antibreathers'', because if we replace
$x_{1}$ by $x_{1}+\frac{\pi}{\alpha}$, then a breather is transformed in its opposite.

Similarly to (\ref{2GS}), it is known that a breather $B=B_{\alpha,\beta}$
satisfies the following elliptic equation on $\mathbb{R}$:
\begin{equation}
B_{xxxx}+5BB_{x}^{2}+5B^{2}B_{xx}+\frac{3}{2}B^{5}-2\left(\beta^{2}-\alpha^{2}\right)\left(B_{xx}+B^{3}\right)+\left(\alpha^{2}+\beta^{2}\right)^{2}B=0.
\end{equation}

Note that in \cite{key-49} a similar form of elliptic equation (with a fourth order derivative) is obtained for a soliton $R_{c,\kappa}$ from \eqref{2GS}. This allows us to consider solitons and breathers at the same  level of regularity, which is the key of the proof made in this article.

This object was first introduced by Wadati \cite{key-31}, and it
was used by Kenig, Ponce and Vega in \cite{key-45} for the ill-posedness for \eqref{2mKdV} for rough data.
Their properties, in particular their stability, are well studied
by Alejo, Muñoz and co-authors \cite{key-1,key-15,key-14,key-17,key-18}.
We know that a breather is orbitally stable in $H^{2}$ \cite{key-1}.
Afterwards, $H^{1}$ orbital stability was proved via Bäcklund transformation \cite{key-18}, and also
$H^{1}$ asymptotic stability, for breathers moving
to the right. Asymptotic stability of breathers in full generality is still an open problem.

When $\alpha \to 0$, $B_{\alpha,\beta}$ tends to a \emph{double-pole
solution} of \eqref{2mKdV}: it is a couple soliton-antisoliton that move
with a constant velocity and that have a repulsive logarithmic interaction
\cite{key-37}. However, this limit, which is somehow at a boundary
between solitons and breathers, is expected to be unstable \cite{key-46}.
We do not consider this object in this paper.

Solitons and breathers are important objects to study because of their stability properties and also because of the soliton-breather resolution. The latter is an important result about the long time dynamics of \eqref{2mKdV}, which asserts that any generic solution can be approached by a sum of solitons and breathers when $t\rightarrow +\infty$. It is established for initial conditions in a weighted Sobolev space in \cite{key-25} (see also Schuur \cite{key-34}) by inverse scattering method; see also \cite{key-34} for the soliton resolution for (KdV).

Given a set of basic objects (solitons and breathers), we consider
a solution that tends to this sum when $t \rightarrow+\infty$, called
\emph{multi-breather}. In \cite{key-49}, we have shown existence,
regularity and uniqueness of multi-breathers of \eqref{2mKdV}. There is also
a formula for multi-breathers of \eqref{2mKdV} obtained by Wadati \cite{key-37}, which was derived as a consequence of the integrability of \eqref{2mKdV}); but is it not
well suited for our purpose.


 Martel, Merle and Tsai \cite{key-8}
proved that a sum of (decoupled and ordered) solitons is orbitally stable in $H^{1}$ for \eqref{2mKdV}, and actually asymptotically stable (in the region $x \ge \delta t$ for $\delta>0$ small). Le Coz \cite{key-55} has established stability of  \eqref{2mKdV} $N$-solitons in $H^N$ by modifying the approach used by Maddocks
and Sachs \cite{key-24} for (KdV). 

Inspired by \cite{key-8}, similar asymptotic stability results where obtained for sums of (decoupled) solitons for various nonlinear dispersive equations: we refer to El Dika \cite{key-64} for the Benjamin-Bona-Mahony equation (BBM), Kenig and Martel \cite{key-65} for the Benjamin-Ono equation (BO), El Dika-Molinet \cite{peakon} for the Camassa-Holm multipeakon, and Côte-Muñoz-Pilod-Simpson \cite{key-63} for the Zakharov-Kuznetsov (ZK) equation.

Because breathers have an $H^{2}$ structure, we prove orbital
stability in $H^{2}$ in this paper for a sum of solitons and breathers.
One of the difficulties is to obtain $H^{2}$ stability results for
solitons too, i.e. to study solitons at a $H^{2}$ level.

\subsection{Main results}
\label{sec:main}

We prove in this article that given any sum of solitons and breathers
with distinct velocities and such that all these velocities except
possibly one are positive, a solution $u$ of \eqref{2mKdV} that is initially
close to this sum in $H^{2}$ stays close to this sum for any time of a considered time interval
up to space translations for solitons or space and phase translations
for breathers. This is orbital stability. Let us make the definition
of orbital stability more precise.

Let $L\in\mathbb{N}$. We consider a set of
$L$ solitons: given $c_{l}^{0}>0$, $\kappa_{l}\in\{-1,1\}$ and
$x_{0,l}^{0}\in\mathbb{R}$ for $1\leq l\leq L$, we set, for $1\leq l\leq L$,
\begin{equation}
R_{l}(t,x):=R_{c_{l}^{0},\kappa_{l}}(t,x;x_{0,l}^{0}).\label{2eq:sol}
\end{equation}

Let $K\in\mathbb{N}$. We consider a set of $K$ breathers: given
$\alpha_{k}>0$, $\beta_{k}>0$ and $x_{1,k}^{0},x_{2,k}^{0}\in\mathbb{R}$
for $1\leq k\leq K$, we set, for $1\leq k\leq K$, 
\begin{equation}
B_{k}(t,x):=B_{\alpha_{k},\beta_{k}}(t,x;x_{1,k}^{0},x_{2,k}^{0}).\label{2eq:br}
\end{equation}

We now define important parameters for each of the objects of the problem.
For $1\leq l\leq L$, the velocity of the $l$th soliton is 
\begin{equation}
v_{l}^{s}:=c_{l}^{0},\label{2eq:vs}
\end{equation}
and for $1\leq k\leq K$, the velocity of the $k$th breather is 

\begin{equation}
v_{k}^{b}:=\beta_{k}^{2}-3\alpha_{k}^{2}.\label{2eq:vb}
\end{equation}

For $1\leq l\leq L$,
the center of the $l$th soliton is
\begin{equation}
x_{l}^{s}(t):=x_{0,l}^{0}+v_{l}^{s}t,\label{2eq:ps}
\end{equation}
and for $1\leq k\leq K$, the center of the $k$th breather is
\begin{equation}
x_{k}^{b}(t):=-x_{2,k}^{0}+v_{k}^{b}t.\label{2eq:pb}
\end{equation}

We set $J:=K+L$ the total number of objects in the problem. We
assume that $J\geq2$, because for $J=1$ the proof is already done
in \cite{key-1} when the only object is a breather, and in \cite{key-6}
when the only object is a soliton.

We assume that the velocities of our objects are all distinct, this
will imply that our objects are far from each other when time is large,
which is essential for our analysis. More precisely, 
\begin{equation}
\forall k\neq k',\quad v_{k}^{b}\neq v_{k'}^{b},\qquad\forall l\neq l',\quad v_{l}^{s}\neq v_{l'}^{s},\qquad\forall k,l,\quad v_{k}^{b}\neq v_{l}^{s}.\label{2diff}
\end{equation}

This allows us to define an increasing function 
\begin{equation}
\underline{v}:\{1,...,J\}\rightarrow\{v_{k}^{b},1\leq k\leq K\}\cup\{v_{l}^{s},1\leq l\leq L\}.
\end{equation}

The set $\{v_{1},...,v_{J}\}$ is thus the set of all the possible
velocities of our objects. We have 
\begin{equation}
v_{1}<v_{2}<...<v_{J}.\label{2eq:vord}
\end{equation}

We define, for $1\leq j\leq J$, $P_{j}$ as the object ($R_{l}$
or $B_{k}$) that corresponds to the velocity $v_{j}$, i.e. if $v_{j}=v_{l}^{s}$,
we set $P_{j}:=R_{l}$, and if $v_{j}=v_{k}^{b}$, we set $P_{j}:=B_{k}$.
So, $P_{1},...,P_{J}$ are the considered objects that are ordered
by increasing velocity.

We denote $x_{j}(t)$ the position (the center of mass) of $P_{j}(t)$.
More precisely, if $P_{j}=R_{l}$, we set $x_{j}(t):=x_{l}^{s}(t)$;
and if $P_{j}=B_{k}$, we set $x_{j}(t):=x_{k}^{b}(t)$.

We set 
\begin{equation}
R:=\sum_{l=1}^{L}R_{l},\quad B:=\sum_{k=1}^{K}B_{k},\quad P:=R+B=\sum_{j=1}^{J}P_{j}.\label{2eq:P}
\end{equation}

We need both notations: indexation by $k$ and $l$,
and indexation by $j$, and we keep these notations to avoid
ambiguity.



The main result that we will prove in this article is the following: a sum of decoupled solitons and breathers, with $v_2>0$ (that is, all travel to the right, but at most one can be static or travel to the left), is orbitally stable.
The precise statement is as follows. 
 
\begin{thm}
\label{2:THM:MAIN}Given a set of solitons and breathers (\ref{2eq:sol}), (\ref{2eq:br})
whose velocities (\ref{2eq:vs}) and (\ref{2eq:vb}) satisfy (\ref{2diff}),
and whose positions are set by (\ref{2eq:ps}) and (\ref{2eq:pb}),
we define the corresponding sum $P$ in (\ref{2eq:P}), and we define
$P_{j}$, $v_{j}$, $x_{j}$ a reindexation of the given set of solitons and
breathers such that (\ref{2eq:vord}).
 We assume that
\begin{equation}
v_{2}>0. \label{2pos}
\end{equation}
Then there exists
$A_{0},\theta_{0},D_{0},a_{0}>0$, constants (depending on $c_l^0$, $\alpha_k$, $\beta_k$, but not on $x_{0,l}^0$, $x_{1,k}^0$ or $x_{2,k}^0$), such that the following is true.
Let $D\geq D_{0}$ and $0\leq a\leq a_{0}$ 
such that 
\begin{equation}
\left\Vert u(0)-P(0)\right\Vert _{H^{2}}\leq a,\quad \text{and} \quad x_{j}(0)>x_{j-1}(0)+2D,\ \text{for all } j=2,\dots,J,\label{2assump:MAIN}
\end{equation}
for a solution $u\in \mathcal{C}(\mathbb{R},H^2(\mathbb{R}))$ of \eqref{2mKdV}. 

Then,
there exist $x_{0,l}(t),x_{1,k}(t),x_{2,k}(t)$ defined for any $t\geq0$
such that 
\begin{equation}
\forall t\geq0,\quad\left\Vert u(t)-\sum_{l=1}^{L}R_{c_{l}^{0},\kappa_{l}}(t,\cdot;x_{0,l}(t))-\sum_{k=1}^{K}B_{\alpha_{k},\beta_{k}}(t,\cdot;x_{1,k}(t),x_{2,k}(t))\right\Vert _{H^{2}}\leq A_{0}\left(a+e^{-\theta_{0}D}\right),
\end{equation}
with
\begin{equation}
\label{2deriv0}
\forall t \geq0,\quad \sum_{l=1}^L \left\vert x_{0,l}'(t) \right\vert + \sum_{k=1}^K \left( \vert x_{1,k}'(t) \vert + \vert x_{2,k}'(t) \vert \right) \leq CA_0 \left( a + e^{- \theta_{0}D}\right),
\end{equation}
for some constant $C>0$.
\end{thm}


%
%

\begin{rem}
Let us stress on the fact that the constant $A_0$ \emph{do not depend} on the translation parameters of the considered objects.
\end{rem}


We also deduce a consequence  from Theorem \ref{2:THM:MAIN} for multi-breathers.

Recall that multi-breathers were defined and contructed in \cite{key-49}, and
their uniqueness was proved there when $v_2>0$. 
They can be also obtained from the formula from \cite{key-37} obtained by inverse scattering method by Wadati.
Let $p$ be a multi-breather associated to $P$ (in the setting $v_2>0$, we know that $p$ is also unique). To emphasize the dependence of $p$ with respect to the parameters,
we will write 
\begin{equation}
p(t,x)=:p(t,x;\alpha_{k},\beta_{k},x_{1,k}^{0},x_{2,k}^{0},\kappa_{l},c_{l}^{0},x_{0,l}^{0}).\label{2eq:mb}
\end{equation}

The same way, to emphasize the dependence of $P$ with respect to
the parameters, we will write
\begin{equation}
P(t,x)=:P(t,x;\alpha_{k},\beta_{k},x_{1,k}^{0},x_{2,k}^{0},\kappa_{l},c_{l}^{0},x_{0,l}^{0}).\label{2eq:sosb}
\end{equation}


Theorem \ref{2:THM:MAIN} can be recast as orbital
stability of a \emph{multi-breather}:

\begin{thm}
\label{2thm:mb}Let $\alpha_{k},\beta_{k},x_{1,k}^{0},x_{2,k}^{0},\kappa_{l},c_{l}^{0},x_{0,l}^{0}$
and $p$ the multi-breather associated to these parameters with notations
as in (\ref{2eq:mb}) given by \cite[Theorem 1.2]{key-49}. We assume
that  \eqref{2diff} holds and $v_{2} > 0$. There exists
$\eta_0>0$ small enough, $C_0>0$ large enough such that the following is true for $0<\eta<\eta_0$. 

Let $u(t)$ be a solution
of \eqref{2mKdV}, such that 
\begin{align} \left\Vert u(0)-p(0;\alpha_{k},\beta_{k},x_{1,k}^{0},x_{2,k}^{0},c_{l}^{0},x_{0,l}^{0})\right\Vert _{H^{2}}\leq \eta. \end{align}
Then there exist $C^1$ functions $x_{0,l}(t),x_{1,k}(t),x_{2,k}(t)$ defined for
any $t\geq 0$ such that 
\begin{equation}
\label{2stab_mb}
\forall t\geq0,\quad\left\Vert u(t)-p\left(t;\alpha_{k},\beta_{k},x_{1,k}(t),x_{2,k}(t),c_{l}^{0},x_{0,l}(t)\right)\right\Vert _{H^{2}}\leq C_0 \eta,
\end{equation}
with \eqref{2deriv0} that is satisfied.
\end{thm}
\begin{rem}
Let us stress on the fact the $C_0$ \emph{do depend} on the translation parameters of the considered solitons and breathers, i.e. on $x_{0,l}^0$, $x_{1,k}^0$ and $x_{2,k}^0$. More precisely, it depends on the time we need to wait until the collisions between the considered solitons and breathers are over. This is a fundamental difference between Theorem \ref{2thm:mb} and Theorem \ref{2:THM:MAIN}. 
In fact, in Theorem \ref{2:THM:MAIN}, we avoid collisions and this is the reason why we get a more uniform result. It is reasonnable to expect that Theorem \ref{2thm:mb} can be improved so that $C_0$ do not depend on translation parameters. 
\end{rem}

%
%


In this paper, we adapt the arguments given by Martel, Merle and Tsai
\cite{key-8} to the context of breathers. To do so, it is needed
to understand the variational structure of breathers, in the same
manner as Weinstein did in \cite{key-6} for (NLS) and \eqref{2mKdV} solitons.
Such results have been obtained by Alejo and Muñoz in \cite{key-1}.
When a soliton is a critical point of a Lyapunov functional at the
$H^{1}$ level, whose Hessian is coercive up to two orthogonal conditions,
a breather is a critical point of a Lyapunov functional at the $H^{2}$
level, whose Hessian is coercive up to three orthogonal conditions.
One important issue that we address is to understand the variational
structure of a soliton at the $H^{2}$ level. We do this by modifying the Lyapunov
functional from \cite{key-1}, and we will also adapt it for a sum
of solitons and breathers. We need to make assumptions on the velocities
of our breathers (recall that the velocity of a soliton is always
positive), because several arguments are based on monotonicity properties, which hold only on the right.

\subsection{Organisation of the proof}

The proof of Theorem \ref{2:THM:MAIN} is based on two results:
a modulation lemma and a bootstrap proposition. We give a detailed outline of both results in Section \ref{2sec:ind}, and give in Section \ref{2sec:orb} the proof of the heart of the argument (Proposition \ref{2prop:ind}), where we complete the bootstrap via a (finite) induction argument on the an improved bound localized on the last $j$ objects.

In Section \ref{2sec:cons}, we prove Theorem
\ref{2thm:mb} as a quick consequence of Theorem \ref{2:THM:MAIN}.

\subsection{Acknowledgments}

The author would like to thank his supervisor Raphaël Côte for suggesting
the idea of the work, for fruitful discussions and his useful advice.

The author would also like to thank Guillaume Ferriere for his useful remarks and suggestions.

\section{Reduction of the proof to an induction}
\label{2sec:ind}

For the proof, we assume the assumption of the Theorem \ref{2:THM:MAIN}
true (i.e. we assume (\ref{2assump:MAIN}) true) for a solution $u(t)$ and a choice of translation parameters, and the goal is to
find the suitable constants $A,\theta,D_{0},a_{0}$ (that do not depend on $u$ nor on translation parameters) so that the Theorem \ref{2:THM:MAIN}
holds.

\subsection{Some useful notations}

We set some useful constants for this paper. We define the worst exponential
decay rate: 
\begin{equation}
\beta:=\min\{\beta_{k},1\leq k\leq K\}\cup\{\sqrt{c_{l}},1\leq l\leq L\},
\end{equation}
and the worst distance between two consecutive velocities: 
\begin{equation}
\tau:=\min\{v_{j+1}-v_{j},1\leq j\leq J-1\}.
\end{equation}


For any $j=2,...,J$, we have that
\begin{equation}
\label{2go}
\forall t\in I,\quad x_{j}(t)-x_{j-1}(t)\geq x_{j}(0)-x_{j-1}(0)+\tau t.
\end{equation}

We introduce general parameters, for $j=1,...,J$, $(a_{j},b_{j})$.
If $P_{j}=B_{k}$ is a breather, we set $(a_{j},b_{j}):=(\alpha_{k},\beta_{k})$.
If $P_{j}=R_{l}$ is a soliton, we set $(a_{j},b_{j}):=(0,\sqrt{c_{l}^{0}})$.

\subsection{Modulation lemma}

We will first state a standard modulation lemma, which can be proved
similarly as the modulation lemma in \cite{key-49}. We need it because
we will construct translations of our objects so that they are near
to $u$ and some orthogonality conditions are satisfied. These orthogonality
conditions will allow us to use coercivity of some quadratic forms
in the following of the proof. 
\begin{lem}
\label{2lem:mod}Let $A'>0$, $\theta>0$, $t'>0$ and 
$y_{1,k}(t)$, $y_{2,k}(t)$, $y_{3,l}(t)$, $y_{0,l}(t)$ defined for $t\in [0,t']$
(with $y_{3,l}(t)>0$) such that $\forall t\in [0,t'],\quad\left|y_{3,l}(t)-c_{l}^{0}\right|\leq\min\{\frac{\tau}{8}\}\cup\{\frac{c_{p}^{0}}{4},1\leq p\leq L\}$.
If $D_{0}$ is large enough and $a_{0}$ is small enough (dependently
on $A'$), there exists a constant $C_{2}>1$ such that the following
holds. Let $u(t)$ be a solution of \eqref{2mKdV} such that for any $t\in [0,t']$,
\begin{equation}
\left\Vert u(t)-\sum_{l=1}^{L}\kappa_lQ_{y_{3,l}(t)}(\cdot+y_{0,l}(t)-c_l^0t)-\sum_{k=1}^{K}B_{\alpha_{k},\beta_{k}}(t,\cdot;y_{1,k}(t),y_{2,k}(t))\right\Vert _{H^{2}}\leq A'\left(a+e^{-\theta D}\right),
\end{equation}
\begin{equation}
\left\Vert u(0)-\sum_{l=1}^{L}\kappa_lQ_{y_{3,l}(0)}(\cdot+y_{0,l}(0))-\sum_{k=1}^{K}B_{\alpha_{k},\beta_{k}}(0,\cdot;y_{1,k}(0),y_{2,k}(0))\right\Vert _{H^{2}}\leq a,
\end{equation}
and if we set $y_{j}(t):=-y_{0,l}(t)+v_l^st$ if $P_{j}=R_{l}$,
and $y_{j}(t):=-y_{2,k}(t)+v_{k}^{b}t$ if $P_{j}=B_{k}$, we have, for any $2\leq j\leq J$,
\begin{equation}
\forall t\in [0,t'],\quad y_{j}(t)-y_{j-1}(t)\geq D,
\end{equation}

then, there exists $\mathcal{C}^{1}$ functions $z_{1,k}(t),z_{2,k}(t),z_{3,l}(t),z_{0,l}(t)$
defined for $t\in [0,t']$ (with $z_{3,l}(t)>0$), such that if we
set 
\begin{equation}
\varepsilon(t):=u(t)-\widetilde{P}(t),
\end{equation}
where 
\begin{equation}
\widetilde{R_{l}}(t,x):=\kappa_lQ_{z_{3,l}(t)}(\cdot+z_{0,l}(t)-c_l^0t)\quad for\,\,1\leq l\leq L,
\end{equation}
\begin{equation}
\widetilde{B_{k}}(t,x):=B_{\alpha_{k},\beta_{k}}(t,x;z_{1,k}(t),z_{2,k}(t))\quad for\,\,1\leq k\leq K,
\end{equation}
\begin{equation}
\widetilde{P_{j}}:=\widetilde{R_{l}}\quad if\,\,P_{j}=R_{l},\quad\widetilde{P_{j}}:=\widetilde{B_{k}}\quad if\,\,P_{j}=B_{k},
\end{equation}
\begin{equation}
\widetilde{R}:=\sum_{l=1}^{L}\widetilde{R_{l}},\quad\widetilde{B}:=\sum_{k=1}^{K}\widetilde{B_{k}},\quad\widetilde{P}:=\widetilde{R}+\widetilde{B}=\sum_{j=1}^{J}\widetilde{P_{j}},
\end{equation}

then, for $t\in [0,t']$, for $l=1,...,L$, for $k=1,...,K$, 
\begin{equation}
\int\widetilde{R_{l}}(t)\varepsilon(t)=\int\partial_{x}\widetilde{R_{l}}(t)\varepsilon(t)=\int\partial_{x_{1}}\widetilde{B_{k}}(t)\varepsilon(t)=\int\partial_{x_{2}}\widetilde{B_{k}}(t)\varepsilon(t)=0.\label{2orth}
\end{equation}

Moreover, for $t\in [0,t']$, we have 
\begin{equation}\begin{aligned}
\Vert\varepsilon(t)\Vert_{H^{2}}+\vert z_{1,k}(t)-y_{1,k}(t)\vert+\vert z_{2,k}(t)-y_{2,k}(t)\vert\\
+\vert z_{3,l}(t)-y_{3,l}(t)\vert+\vert z_{0,l}(t)-y_{0,l}(t)\vert &\leq C_{2}A'\left(a+e^{-\theta D}\right),
\end{aligned}\end{equation}
and 
\begin{equation}
\Vert\varepsilon(0)\Vert_{H^{2}}+\vert z_{1,k}(0)-y_{1,k}(0)\vert+\vert z_{2,k}(0)-y_{2,k}(0)\vert+\vert z_{3,l}(0)-y_{3,l}(0)\vert+\vert z_{0,l}(0)-y_{0,l}(0)\vert\leq C_{2}a,
\end{equation}
and for any $t\in [0,t']$, $\left(z_{1,k}(t),z_{2,k}(t),z_{3,l}(t),z_{0,l}(t)\right)\in\mathbb{R}^{2K+2L}$
is unique such that (\ref{2orth}) is satisfied and $\left(z_{1,k}(t),z_{2,k}(t),z_{3,l}(t),z_{0,l}(t)\right)$
is in a suitable neighbourhood of $\left(y_{1,k}(t),y_{2,k}(t),y_{3,l}(t),y_{0,l}(t)\right)$
that depends only on $A'\left(a+e^{-\theta D}\right)$.

We set, for $t\in [0,t']$, $z_j(t):=z_{l}^{s}(t):=-z_{0,l}(t)+v_l^st$
 if $P_{j}=R_{l}$, and $z_j(t):=z_{k}^{b}(t):=-z_{2,k}(t)+v_{k}^{b}t$
 if $P_{j}=B_{k}$.

For $D_{0}$ large enough and $a_{0}$ small enough, if we assume
that 
\begin{equation}
\forall t\in [0,t'],\quad z_{j}(t)-z_{j-1}(t)\geq D,\label{2forz}
\end{equation}
(note that (\ref{2forz}) is a consequence of $\forall t\in [0,t'],\quad y_{j}(t)-y_{j-1}(t)\geq2D$ if $a_0$ is small enough and $D_0$ is large enough),
then for any $t\in [0,t']$, we have that

for $k=1,...,K$, 
\begin{equation}
\vert z_{1,k}'(t)\vert+\vert z_{2,k}'(t)\vert\leq C_{2}\left(\int e^{-\frac{\beta}{2}\left|x-z_{k}^{b}(t)\right|}\varepsilon^{2}\right)^{1/2}+C_{2}e^{-\frac{\beta D}{8}},\label{2mod1}
\end{equation}

for $l=1,...,L$, 
\begin{equation}
\vert z_{3,l}'(t)\vert+\vert z_{0,l}'(t)\vert\leq C_{2}\left(\int e^{-\frac{\beta}{2}\left|x-z_{l}^{s}(t)\right|}\varepsilon^{2}\right)^{1/2}+C_{2}e^{-\frac{\beta D}{8}}.\label{2mod2}
\end{equation}
\end{lem}

\begin{rem}
We will also use generalized notations for $y_{1,k}(t)$, $y_{2,k}(t)$ and other object-specific notations, in the lemma above.

For $j=1,...,J$, either if $P_j=R_l$ is a soliton, we denote
\begin{align}
y_{1,j}^*(t):=y_{3,l}(t),\quad z_{1,j}^*(t):=z_{3,l}(t),\quad y_{2,j}^*(t):=y_{0,l}(t),\quad z_{2,j}^*(t):=z_{0,l}(t),
\end{align}
or if $P_j=B_k$ is a breather, we denote
\begin{align}
y_{1,j}^*(t):=y_{1,k}(t),\quad z_{1,j}^*(t):=z_{1,k}(t),\quad y_{2,j}^*(t):=y_{2,k}(t),\quad z_{2,j}^*(t):=z_{2,k}(t).
\end{align}
\end{rem}

\begin{proof}[Proof]
See \cite[Lemma 2.8]{key-49}, for the proof of a similar result. We also refer to \cite{key-8, key-1,CH}.
\end{proof}


\subsection{Bootstrap}

Given Lemma \ref{2lem:mod}, we reduce the proof of Theorem \ref{2:THM:MAIN}
to the following bootstrap proposition (we will use the notations given in Section \ref{sec:main}; in particular the position $x_j(t)$ of $P_j(t)$ is defined there): 
\begin{prop}
\label{2prop:boot}There exists $A\geq C_{2}$ and $\theta>0$ such
that, if $D_{0}$ is large enough and $a_{0}$ is small enough such
that 
\begin{equation}
C_{2}a\leq A\left(a+e^{-\theta D}\right)\leq\min\{\frac{\tau}{8}\}\cup\{\frac{c_{p}^{0}}{4},1\leq p\leq L\},
\end{equation}
for $t^*>0$, we assume that there exist $\mathcal{C}^1$ functions $x_{1,k}(t),x_{2,k}(t),c_{l}(t),x_{0,l}(t)\in\mathbb{R}$
(with $c_{l}(t)>0$) defined for $t\in [0,t^*]$ such that, if we
denote 
\begin{equation}
\forall t\in [0,t^*],\quad\varepsilon(t):=u(t)-\sum_{l=1}^{L}\widetilde{R_{l}}(t)-\sum_{k=1}^{K}\widetilde{B_{k}}(t),
\end{equation}
where $u$ is a solution of \eqref{2mKdV}, 
\begin{equation}
\forall t\in [0,t^*],\quad\widetilde{R_{l}}(t):=\kappa_lQ_{c_l(t)}(\cdot+x_{0,l}(t)-c_l^0t)\quad\text{for }\,1\leq l\leq L,
\end{equation}
and
\begin{equation}
\forall t\in [0,t^*],\quad\widetilde{B_{k}}(t):=B_{\alpha_{k},\beta_{k}}(t,\cdot;x_{1,k}(t),x_{2,k}(t)),\quad\text{for }\,1\leq k\leq K,
\end{equation}
$\widetilde{P_{j}}:=\widetilde{R_{l}}$ if $P_{j}=R_{l}$ and $\widetilde{P_{j}}:=\widetilde{B_{k}}$
if $P_{j}=B_{k}$, and $\widetilde{x_{l}^{s}}(t):=-x_{0,l}(t)+v_l^st$,
$\widetilde{x_{k}^{b}}(t):=-x_{2,k}(t)+v_{k}^{b}t$, $\widetilde{x_{j}}(t):=\widetilde{x_{l}^{s}}(t)$
if $P_{j}=R_{l}$, and $\widetilde{x_{j}}(t):=\widetilde{x_{k}^{b}}(t)$
if $P_{j}=B_{k}$,

and if we assume that 
\begin{equation}
\forall 1\leq j\leq J-1,\quad x_{j+1}(0)-x_{j}(0)\geq 2D,\label{2cond0}
\end{equation}
\begin{equation}
\forall t\in [0,t^*],\quad\Vert\varepsilon(t)\Vert_{H^{2}}\leq A\left(a+e^{-\theta D}\right),\quad\Vert\varepsilon(0)\Vert_{H^{2}}\leq C_{2}a,\label{2cond1}
\end{equation}
\begin{equation}
\forall t\in [0,t^*],\quad\vert c_{l}(t)-c_{l}^{0}\vert\leq A\left(a+e^{-\theta D}\right),\quad\vert c_{l}(0)-c_{l}^{0}\vert\leq C_{2}a,\label{2cond2}
\end{equation}
\begin{equation}
\sum_{l=1}^L\vert x_{0,l}(0)+x_{0,l}^0\vert+\sum_{k=1}^K\left(\vert x_{1,k}(0)-x_{1,k}^0\vert+\vert x_{2,k}(0)-x_{2,k}^0\vert\right)\leq Ca,\label{2cond}
\end{equation}
\begin{equation}
\forall t\in [0,t^*],\quad\sum_{l=1}^L\left(\vert c_l'(t)\vert+\vert x_{0,l}'(t)\vert\right)+\sum_{k=1}^K\left(\vert x_{1,k}'(t)\vert +\vert x_{2,k}'(t)\vert\right)\leq CA(a+e^{-\theta D}),\label{2cond3}
\end{equation}
where $C>0$ is a large enough constant,
and
\begin{equation}
\forall t\in [0,t^*],\quad\int\widetilde{B_{k}}_{1}(t)\varepsilon(t)=\int\widetilde{B_{k}}_{2}(t)\varepsilon(t)=\int\widetilde{R_{l}}(t)\varepsilon(t)=\int\widetilde{R_{l}}_{x}(t)\varepsilon(t)=0,\label{2cond4}
\end{equation}

then 
\begin{equation}
\forall t\in [0,t^*],\quad\Vert\varepsilon(t)\Vert_{H^{2}}\leq\frac{A}{2}\left(a+e^{-\theta D}\right),\label{2result1}
\end{equation}
\begin{equation}
\forall t\in [0,t^*],\quad\vert c_{l}(t)-c_{l}^{0}\vert\leq\frac{A}{2}\left(a+e^{-\theta D}\right).\label{2result2}
\end{equation}

\end{prop}

%
%

\begin{rem}
We will  use generalized notations for $x_{1,k}(t)$, $x_{2,k}(t)$, $c_l(t)$ and $x_{0,l}(t)$.

For $j=1,...,J$, either if $P_j=R_l$ is a soliton, we denote
\begin{align}
x_{1,j}^*(t):=c_l(t),\quad x_{2,j}^*(t):=x_{0,l}(t),
\end{align}
or if $P_j=B_k$ is a breather, we denote
\begin{align}
x_{1,j}^*(t):=x_{1,k}(t),\quad x_{2,j}^*:=x_{2,k}(t).
\end{align}

We denote
\begin{align}
\widetilde{P}(t):=\sum_{j=1}^J\widetilde{P_j}(t)=\sum_{k=1}^K\widetilde{B_k}(t)+\sum_{l=1}^L\widetilde{R_l}(t).
\end{align}
\end{rem}

\begin{rem}
\label{rem:farness}
We note that, for any $j=1,...,J-1$, \eqref{2cond0}
and \eqref{2cond}, assuming that $a_0$ is chosen small enough and $D_0$ is chosen large enough, imply that
\begin{align}
\widetilde{x_{j+1}}(0)-\widetilde{x_{j}}(0)\geq D.\label{2cons2}
\end{align}

From definitions of $\widetilde{x_j}(t)$, of $\tau$ and \eqref{2cond3}, assuming that $a_0$ is chosen small enough and $D_0$ is chosen large enough, we deduce that for any $t\in[0,t^*]$,
\begin{align}
\forall 1\leq j\leq J-1,\quad\widetilde{x_{j+1}}'(t)-\widetilde{x_j}'(t)\geq\frac{\tau}{2},\label{2cons3}
\end{align}
and
\begin{equation}\begin{aligned}
\forall 2\leq j\leq J,\quad\frac{v_2}{2}\leq \widetilde{x_j}'(t)\leq 2v_J.\label{2cons4}
\end{aligned}\end{equation}

From \eqref{2cons2} and \eqref{2cons3}, we may deduce that for any $t\in[0,t^*]$,
\begin{align}
\forall 1\leq j\leq J,\quad\widetilde{x_{j+1}}(t)-\widetilde{x_j}(t)\geq D+\frac{\tau}{2}t.\label{2cons5}
\end{align}

\end{rem}

The proof of this proposition will be the goal of the following. The
proof of Theorem \ref{2:THM:MAIN} then follows from a continuity argument. 
\begin{proof}[Proof of Theorem \ref{2:THM:MAIN} assuming Proposition \ref{2prop:boot}]
We take $A,\theta,D_{0},a_{0}$ that work for Proposition \ref{2prop:boot}, we
will show that they will also work for Theorem \ref{2:THM:MAIN}. We take $D\geq D_{0}$
and $0\leq a\leq a_{0}$ and we assume that (\ref{2assump:MAIN}) is
true for a solution $u$ of \eqref{2mKdV}. This implies \eqref{2cond0}.

We assume that $a_{0}$ is small enough and $D_{0}$ is large enough
such that $A\left(a+e^{-\theta D}\right)\leq\min\{\frac{\tau}{8}\}\cup\{\frac{c_{p}^{0}}{4},1\leq p\leq L\}$.

Because $\frac{A}{C_{2}}>1$, by continuity, there exists $t_1>0$ such that 
\begin{equation}
\forall t\in [0,t_1],\quad\left\Vert u(t)-P(t)\right\Vert _{H^{2}}\leq\frac{A}{C_{2}}\left(a+e^{-\theta D}\right).
\end{equation}

Of course, we have that 
\begin{equation}
\forall t\in [0,t_1],\quad x_{j}(t)-x_{j-1}(t)\geq2D+\tau t.
\end{equation}

We apply Lemma \ref{2lem:mod} on $[0,t_1]$ with $A'=\frac{A}{C_{2}}$.
We take $D_{0}$ larger and $a_{0}$ smaller if needed. So, there
exist $\mathcal{C}^{1}$ functions $x_{1,k}(t),x_{2,k}(t),c_{l}(t),x_{0,l}(t)\in\mathbb{R}$
with $c_{l}(t)>0$ defined for $t\in [0,t_1]$, such that if we set
\begin{equation}
\varepsilon(t):=u(t)-\widetilde{P}(t),
\end{equation}
with the same notations as usual, we have that
\begin{equation}
\forall t\in [0,t_1],\quad\Vert\varepsilon(t)\Vert_{H^{2}}\leq A\left(a+e^{-\theta D}\right),\quad\Vert\varepsilon(0)\Vert_{H^{2}}\leq C_{2}a,
\end{equation}
\begin{equation}
\forall t\in [0,t_1],\quad\vert c_{l}(t)-c_{l}^{0}\vert\leq A\left(a+e^{-\theta D}\right),\quad\vert c_{l}(0)-c_{l}^{0}\vert\leq C_{2}a,
\end{equation}
\begin{equation}
\sum_{l=1}^L\vert x_{0,l}(0)+x_{0,l}^0\vert+\sum_{k=1}^K\left(\vert x_{1,k}(0)-x_{1,k}^0\vert+\vert x_{2,k}(0)-x_{2,k}^0\vert\right)\leq Ca,
\end{equation}
\begin{equation}
\forall t\in [0,t_1],\quad\sum_{l=1}^L\left(\vert c_l'(t)\vert+\vert x_{0,l}'(t)\vert\right)+\sum_{k=1}^K\left(\vert x_{1,k}'(t)\vert +\vert x_{2,k}'(t)\vert\right)\leq CA(a+e^{-\theta D}),
\end{equation}
if $\theta>0$ is chosen small enough, and
\begin{equation}
\forall t\in [0,t_1],\quad\int\widetilde{B_{k}}_{1}(t)\varepsilon(t)=\int\widetilde{B_{k}}_{2}(t)\varepsilon(t)=\int\widetilde{R_{l}}(t)\varepsilon(t)=\int\widetilde{R_{l}}_{x}(t)\varepsilon(t)=0.
\end{equation}

Remember that \eqref{2cons5} for any $t\in[0,t_1]$ is a consequence of what is written above.

Let $I^*\subset\mathbb{R}_+$ be the supremum of intervals $I'\supset [0,t_1]$ for inclusion in the set of intervals of $\mathbb{R}_+$ such that the $C^{1}$
functions $x_{1,k}(t)$, $x_{2,k}(t)$, $c_{l}(t)$, $x_{0,l}(t)$ may be extended
on $I'$ and such that \eqref{2cond0}, (\ref{2cond1}), (\ref{2cond2}), \eqref{2cond}, (\ref{2cond3})
and (\ref{2cond4}) are still satisfied for any $t\in I'$.


In order to make the definition above licit, we need to point out the following fact: by uniqueness in Lemma \ref{2lem:mod}, we find that if we have two
extensions on $I_2$ and $I_3$, then 	we find that those extensions coincide on $I_2\cap I_3$ (we remind that $[0,t_1]\subset I_2\cap I_3$). This is why, if we have suitable extensions on $I_2$ and $I_3$, then we have a suitable extension on $I_2\cup I_3$. And, the supremum $I^*$ is simply the union of all the possible extensions $I'$ and the implicit functions may be extended on $I^*$.



When \eqref{2cond0}, (\ref{2cond1}), (\ref{2cond2}), \eqref{2cond}, (\ref{2cond3}) and (\ref{2cond4})
are true for $x_{1,k}(t),x_{2,k}(t),c_{l}(t),x_{0,l}(t)$ with $I=[0,t_2]\subset I^*$ (where $t_2\geq t_1$), we may extend
these implicit functions a bit further (in a random way, that we denote $y_{1,k}(t)$, $y_{2,k}(t)$, $y_{3,l}(t)$ and $y_{0,l}(t)$, and we denote $y_j$ their positions) and have \eqref{2cond0}, (\ref{2cond1}), (\ref{2cond2}), \eqref{2cond},
(\ref{2cond3}) and (\ref{2cond4}) that are still satisfied, but on
the extended interval. The extended interval will be an interval of the form $[0,t_2']$, where $t_2'>t_2$. We do it in the following way. First, we apply
Proposition \ref{2prop:boot} on $I$, and that makes
(\ref{2cond1}) and (\ref{2cond2}) a bit improved on $I$
and become (\ref{2result1}) and (\ref{2result2}) on $I$.
And so,  (\ref{2cond1})
and (\ref{2cond2}) are satisfied on an extension of $I$.
After application of Lemma \ref{2lem:mod} with $A'=A$ (where the notation $A'$ is from Lemma \ref{2lem:mod}), that we may apply thanks to \eqref{2cons5} (more precisely, from Remark \ref{rem:farness}, \eqref{2cons5} is satisfied on $I$, from what for any $1\leq j\leq J-1$, $\widetilde{y_{j+1}}(t)-\widetilde{y_j}(t)\geq D$ is true on an extension of $I$)
we see that (\ref{2cond4}) can be also extended (after modification of the implicit functions on the extension of $I$: we now denote them $x_{1,k}(t)$, $x_{2,k}(t)$, $x_{3,l}(t)$ and $x_{0,l}(t)$, as they coincide with implicit functions given at the beginning of this paragraph on $I$ from uniqueness in Lemma \ref{2lem:mod}).  Note that after this modification, it is needed to reconsider all the previous sentences of this paragraph, in order to take a smaller extension if needed and have \eqref{2cond1} and \eqref{2cond2} satisfied. Because of Remark \ref{rem:farness},  \eqref{2cons5} is satisfied on $I$, this is why for any $1\leq j\leq J-1$, $\widetilde{x_{j+1}}(t)-\widetilde{x_j}(t)\geq D$ is satisfied on an extension of $I$ (eventually smaller, but of the form $[0,t_2']$ as specified at the beginning of the paragraph).
This is why, we may apply consequences \eqref{2mod1} and \eqref{2mod2} of Lemma  \ref{2lem:mod} and obtain that \eqref{2cond3} is also satisfied on this extension of $I$.
Thus, we may
find an extension $[0,t_2']$ (with $t_2'>t_2$) of $I$ and extensions of $x_{1,k}(t),x_{2,k}(t),c_{l}(t),x_{0,l}(t)$
on this extension such that \eqref{2cond0} (\ref{2cond1}), (\ref{2cond2}), \eqref{2cond} (\ref{2cond3}) 
and (\ref{2cond4}) are satisfied.

We deduce that the interval $I^*$ is necesseraly an open subset of $\mathbb{R}_+$. Let us prove that $I^*$ is also a closed subset of $\mathbb{R}_+$.
This will allow us to conclude that $I^*=\mathbb{R}_+$.

If $I^*$ is not closed, then it is $[0,t_2)$, where $t_2>0$. Let $(T_n)$ be a sequence of points of $I^*$ that converges to $t_2$. Then, $(x_{1,k}(T_n))$, $(x_{2,k}(T_n))$, $(c_l(T_n))$ and $(x_{0,l}(T_n))$ are Cauchy sequences because of  \eqref{2cond3}. Thus, they converge and $x_{1,k}(t)$, $x_{2,k}(t)$, $c_l(t)$ and $x_{0,l}(t)$ may be extended continuously in $t=t_2$. 
By continuity, it is clear that \eqref{2cond1}, \eqref{2cond2}, \eqref{2cond4} and \eqref{2cons5} are still satisfied in $t=t_2$. This is why, we may apply Lemma \ref{2lem:mod} on $[0,t_2]$. From uniqueness in Lemma \ref{2lem:mod}, we find that $x_{1,k}(t)$, $x_{2,k}(t)$, $c_l(t)$ and $x_{0,l}(t)$ are $\mathcal{C}^1$ on $[0,t_2]$. Thus, \eqref{2cond3} is also satisfied in $t=t_2$. That contradicts the maximality of $I^*$. Thus, $I^*$ is closed in $\mathbb{R}_+$.


So, we deduce that we have $C^1$ functions $x_{1,k}(t),x_{2,k}(t),c_{l}(t),x_{0,l}(t)$
defined for any $t\geq0$ such that, 
\begin{equation}
\forall t\geq 0,\quad\Vert\varepsilon(t)\Vert_{H^{2}}\leq A\left(a+e^{-\theta D}\right),
\end{equation}
\begin{equation}
\forall t\geq 0,\quad\vert c_{l}(t)-c_{l}^{0}\vert\leq A\left(a+e^{-\theta D}\right),
\end{equation}
\begin{equation}
\forall t\geq 0,\quad\sum_{l=1}^L\left(\vert c_l'(t)\vert+\vert x_{0,l}'(t)\vert\right)+\sum_{k=1}^K\left(\vert x_{1,k}'(t)\vert +\vert x_{2,k}'(t)\vert\right)\leq CA(a+e^{-\theta D}),\label{2derivs}
\end{equation}
where $C>0$ is a large enough constant, and

\begin{equation}
\forall t\geq 0,\quad\int\widetilde{B_{k}}_{1}(t)\varepsilon(t)=\int\widetilde{B_{k}}_{2}(t)\varepsilon(t)=\int\widetilde{R_{l}}(t)\varepsilon(t)=\int\widetilde{R_{l}}_{x}(t)\varepsilon(t)=0.
\end{equation}

And we set 
\begin{equation}
w(t):=u(t)-\sum_{l=1}^{L}\kappa_lQ_{c_{l}^{0}}(\cdot+x_{0,l}(t)-c_{l}^0t)-\sum_{k=1}^{K}B_{\alpha_{k},\beta_{k}}(t,\cdot;x_{1,k}(t),x_{2,k}(t)).
\end{equation}

We need to bound $w$ to finish the proof. For $t\geq0$, we use triangular inequality
and we bound the norm between two ground states centered at a same
point, 
\begin{align}
\Vert w(t)\Vert_{H^{2}} & \leq\Vert\varepsilon(t)\Vert_{H^{2}}+\sum_{l=1}^{L}\left\Vert Q_{c_{l}(t)}-Q_{c_{l}^{0}}\right\Vert _{H^{2}}\nonumber \\
 & \leq\Vert\varepsilon(t)\Vert_{H^{2}}+C\sum_{l=1}^{L}\vert c_{l}(t)-c_{l}^{0}\vert\nonumber \\
 & \leq CA\left(a+e^{-\theta D}\right),\label{2eq:-23}
\end{align}
and this is exactly what we wanted to prove. 
Moreover \eqref{2deriv0} is a straightforward consequence of \eqref{2derivs}.
\end{proof}
Hence, we are left to prove Proposition \ref{2prop:boot}.

\subsection{Proof by induction}

We will prove Proposition \ref{2prop:boot} by induction. More precisely,
we want to find $A>C_{2},\theta,D_{0},a_{0}$ so that the Proposition \ref{2prop:boot}
holds for any $t^*>0$.

We assume that there exists functions $x_{1,k}(t),x_{2,k}(t),c_{l}(t),x_{0,l}(t)\in\mathbb{R}$
defined for $t\in [0,t^*]$ such that, with notations of Proposition
\ref{2prop:boot}, 
\begin{equation}
\forall t\in [0,t^*],\quad\Vert\varepsilon(t)\Vert_{H^{2}}\leq A(a+e^{-\theta D}),\quad\Vert\varepsilon(0)\Vert_{H^{2}}\leq C_{2}a,\label{2eq:first}
\end{equation}
\begin{equation}
\forall t\in [0,t^*],\quad\vert c_{l}(t)-c_{l}^{0}\vert\leq A(a+e^{-\theta D})\leq\min\left\{ \frac{\tau}{8}\right\} \cup\left\{ \frac{c_{p}^{0}}{4},0\leq p\leq L\right\} ,\quad\vert c_{l}(0)-c_{l}^{0}\vert\leq C_{2}a,\label{2eq:second}
\end{equation}
\begin{equation}
\forall t\in [0,t^*],\quad\sum_{l=1}^L\left(\vert c_l'(t)\vert+\vert x_{0,l}'(t)\vert\right)+\sum_{k=1}^K\left(\vert x_{1,k}'(t)\vert +\vert x_{2,k}'(t)\vert\right)\leq CA(a+e^{-\theta D}),
\end{equation}
\begin{equation}
\forall t\in [0,t^*],\quad\int\widetilde{B_{k}}_{1}(t)\varepsilon(t)=\int\widetilde{B_{k}}_{2}(t)\varepsilon(t)=\int\widetilde{R_{l}}(t)\varepsilon(t)=\int\widetilde{R_{l}}_{x}(t)\varepsilon(t)=0,\label{2orth-cond}
\end{equation}
as well as \eqref{2cond} and \eqref{2cond0}.

As in Remark \ref{rem:farness}, we deduce that \eqref{2cons3}, \eqref{2cons4} and \eqref{2cons5} are satisfied.

And the goal is to improve inequalities (\ref{2eq:first}) and (\ref{2eq:second}). 

We define the average between positions of two consecutive objects.
For $j=3,...,J$, we set 
\begin{equation}
\label{2m_j}
\forall t\in [0,t^*],\quad m_{j}(t):=\frac{\widetilde{x_{j-1}}(t)+\widetilde{x_{j}}(t)}{2},
\end{equation}
and we set 
\begin{equation}
\label{2m_2}
\forall t\in [0,t^*],\quad m_{2}(t):=\frac{\widetilde{x_{1}}(0)+\widetilde{x_{2}}(0)}{2}+\int_{0}^{t}\max\left(\frac{\widetilde{x_{1}}'(s)+\widetilde{x_{2}}'(s)}{2},\frac{\widetilde{x_{2}}'(s)}{2}\right)ds.
\end{equation}
With these definitions, we make sure that for any $j=2,...,J$, $m_j'(t)>0$, even if $v_1<0$.

By (\ref{2cons3}), for $j\geq3$, 
\begin{equation}
\label{2diff_vit}
\forall t\in [0,t^*],\quad\widetilde{x_{j}}'(t)-m_{j}'(t)\geq\frac{\tau}{4},\quad m_{j}'(t)-\widetilde{x_{j-1}}'(t)\geq\frac{\tau}{4},
\end{equation}
and for $j=2$, we have 
\begin{equation}
\forall t\in [0,t^*],\quad m_{2}'(t)=\max\left(\frac{\widetilde{x_{1}}'(t)+\widetilde{x_{2}}'(t)}{2},\frac{\widetilde{x_{2}}'(t)}{2}\right),
\end{equation}
and so, 
\begin{equation}
\forall t\in [0,t^*],\quad m_{2}'(t)-\widetilde{x_{1}}'(t)\geq\frac{\tau}{4},\quad\widetilde{x_{2}}'(t)-m_{2}'(t)\geq\min\left(\frac{v_{2}}{4},\frac{\tau}{4}\right).
\end{equation}

This is why, we set 
\begin{equation}
\zeta:=\min\left(\frac{v_{2}}{4},\frac{\tau}{4}\right),
\end{equation}
a constant that depends only on problem data, and so for any $j\geq 2$,
\begin{equation}
\forall t\in [0,t^*],\quad\widetilde{x_{j}}'(t)-m_{j}'(t)\geq\zeta,\quad m_{j}'(t)-\widetilde{x_{j-1}}'(t)\geq\zeta.\label{2new-minor}
\end{equation}

The latter implies that $\forall j\geq2,\quad\forall t\in [0,t^*],\quad\widetilde{x_{j-1}}(t)<m_{j}(t)<\widetilde{x_{j}}(t)$,
and we may deduce by integration and by (\ref{2cons5}) and (\ref{2new-minor})
that 
\begin{align}
\widetilde{x_{j}}(t)-m_{j}(t) & =\widetilde{x_{j}}(0)-m_{j}(0)+\int_{0}^{t}\left(\widetilde{x_{j}}'(s)-m_{j}'(s)\right)ds\nonumber \\
 & =\frac{\widetilde{x_{j}}(0)-\widetilde{x_{j-1}}(0)}{2}+\int_{0}^{t}\left(\widetilde{x_{j}}'(s)-m_{j}'(s)\right)ds\nonumber \\
 & \geq\frac{D}{2}+\zeta t,\label{2eq:-22}
\end{align}
and similarly, 
\begin{equation}
\label{2e:-22-2}
m_{j}(t)-\widetilde{x_{j-1}}(t)\geq\frac{D}{2}+\zeta t.
\end{equation}

We have that (the $m_{j}$ are chosen for that) for any $j\geq2$,
\begin{equation}
\forall t\in [0,t^*],\quad2v_{J}\geq\widetilde{x_{J}}'(t)\geq m_{j}'(t)\geq\frac{v_{2}}{4}\geq\zeta.
\end{equation}


We will reason by induction in order to improve (\ref{2eq:first}) and (\ref{2eq:second}). For this, we introduce a cut-off function.

Let $\sigma>0$ be a constant small enough for which the conditions
will be fixed in the following of the proof.

We denote: 
\begin{equation}
\Psi(x):=\frac{2}{\pi}\arctan\left(\exp\left(\sqrt{\sigma}x/2\right)\right).
\end{equation}

By direct calculations, 
\begin{equation}
\Psi'(x)=\frac{\sqrt{\sigma}}{2\pi\cosh\left(\sqrt{\sigma}x/2\right)},
\end{equation}
and so, 
\begin{equation}
\label{2der_exp}
\vert\Psi'(x)\vert\leq C\exp\left(-\sqrt{\sigma}\vert x\vert/2\right).
\end{equation}

We have the following properties: $\lim_{+\infty}\Psi=1$, $\lim_{-\infty}\Psi=0$,
for all $x\in\mathbb{R}$ $\Psi(-x)+\Psi(x)=1$, $\Psi'(x)>0$, $\vert\Psi''(x)\vert\leq\frac{\sqrt{\sigma}}{2}\vert\Psi'(x)\vert$,
$\vert\Psi'''(x)\vert\leq\frac{\sqrt{\sigma}}{2}\vert\Psi''(x)\vert$,
$\vert\Psi'(x)\vert\leq\frac{\sqrt{\sigma}}{2}\Psi$ and $\vert\Psi'(x)\vert\leq\frac{\sqrt{\sigma}}{2}\left(1-\Psi\right)$.

We define cut-off functions filtering $P_j$ and all the objects faster than $P_j$: for $j=2,...,J,$ 
\begin{equation}
\Phi_{j}(t,x):=\Psi(x-m_{j}(t)).
\end{equation}

We have: 
\begin{equation}
\left(\Phi_{j}\right)_{t}=-m_{j}'\left(\Phi_{j}\right)_{x}.
\end{equation}

We may extend this definition to $j=1$ and $j=J+1$ in the following
way: $\Phi_{1}:=1$ and $\Phi_{J+1}:=0$.

In order to prove Proposition \ref{2prop:boot} by induction, we will
find an increasing sequence $(Z_{j})_{j=1,...,J+1}$ such that $Z_{1}:=2$
and $Z_{J+1}:=+\infty$ and such that we will be able to prove the
following proposition for any $j=1,...,J$. The goal is to obtain
inequalities of Proposition \ref{2prop:boot}  with better constants.
So, in order to achieve this, we do the following induction: we suppose
that localized inequalities around $P_{j+1},...,P_{J}$ are obtained
with strongly improved constants (constants divided by $Z_{j+1}$),
and we deduce from them localized inequalities around $P_{j},...,P_{J}$
with improved constants, but a little bit less improved than earlier
(constants divided by $Z_{j}$). We will also assume the bootstrap
assumption. This induction is sufficient, because it starts from an
assumption on an empty set of objects and ends with a conclusion with
inequalities localized around all the objects, i.e. global (not localized at all).
\begin{prop}
\label{2prop:ind}Assuming that 
\begin{equation}
\forall t\in [0,t^*],\quad\int\left(\varepsilon^{2}+\varepsilon_{x}^{2}+\varepsilon_{xx}^{2}\right)\Phi_{j+1}\leq\left(\frac{A}{Z_{j+1}}\right)^{2}\left(a^{2}+e^{-2\theta D}\right),\label{2epsilon-assump}
\end{equation}
and for any $j'\geq j+1$ such that $P_{j'}=R_{l}$ is a soliton,
\begin{equation}
\forall t\in [0,t^*],\quad\vert c_{l}(t)-c_{l}(0)\vert\leq\left(\frac{A}{Z_{j+1}}\right)^{2}\left(a^{2}+e^{-2\theta D}\right),\label{2velo-assump}
\end{equation}
and for any $j'\geq j+1$ such that $P_{j'}$ is a breather, 
\begin{equation}
\forall t\in [0,t^*],\quad\left|\int\widetilde{P_{j'}}\varepsilon(t)-\int\widetilde{P_{j'}}\varepsilon(0)\right|\leq\left(\frac{A}{Z_{j+1}}\right)^{2}\left(a^{2}+e^{-2\theta D}\right),\label{2br-m-assump}
\end{equation}
and 
\begin{equation}
\forall t\in [0,t^*],\quad\left|\int\left[\widetilde{P_{j'}}_{xx}+\widetilde{P_{j'}}^{3}\right]\varepsilon(t)-\int\left[\widetilde{P_{j'}}_{xx}+\widetilde{P_{j'}}^{3}\right]\varepsilon(0)\right|\leq\left(\frac{A}{Z_{j+1}}\right)^{2}\left(a^{2}+e^{-2\theta D}\right),\label{2br-e-assump}
\end{equation}

we have that 
\begin{equation}
\forall t\in [0,t^*],\quad\int\left(\varepsilon^{2}+\varepsilon_{x}^{2}+\varepsilon_{xx}^{2}\right)\Phi_{j}\leq\left(\frac{A}{Z_{j}}\right)^{2}\left(a^{2}+e^{-2\theta D}\right),\label{2epsilon-ccl}
\end{equation}
and if $P_{j}=R_{l}$ is a soliton, 
\begin{equation}
\forall t\in [0,t^*],\quad\vert c_{l}(t)-c_{l}(0)\vert\leq\left(\frac{A}{Z_{j}}\right)^{2}\left(a^{2}+e^{-2\theta D}\right),\label{2velo-ccl}
\end{equation}
and if $P_{j}$ is a breather, 
\begin{equation}
\forall t\in [0,t^*],\quad\left|\int\widetilde{P_{j}}\varepsilon(t)-\int\widetilde{P_{j}}\varepsilon(0)\right|\leq\left(\frac{A}{Z_{j}}\right)^{2}\left(a^{2}+e^{-2\theta D}\right),\label{2eq:mass_var}
\end{equation}
and 
\begin{equation}
\forall t\in [0,t^*],\quad\left|\int\left[\widetilde{P_{j}}_{xx}+\widetilde{P_{j}}^{3}\right]\varepsilon(t)-\int\left[\widetilde{P_{j}}_{xx}+\widetilde{P_{j}}^{3}\right]\varepsilon(0)\right|\leq\left(\frac{A}{Z_{j}}\right)^{2}\left(a^{2}+e^{-2\theta D}\right).\label{2eq:energ_var}
\end{equation}
\end{prop}

\begin{rem}
\label{2rem:P_j}
For $j=1,...,J$, we may denote by $\mathcal{P}_{j}$ the following
assertion: (\ref{2epsilon-ccl}), (\ref{2velo-ccl}), (\ref{2eq:mass_var})
and (\ref{2eq:energ_var}). The Proposition \ref{2prop:ind} may be
reformulated in the following way:

There exists an increasing sequence $(Z_{j})_{j=1,...,J+1}$ with
$Z_{1}=2$ and $Z_{J+1}=+\infty$, $A$ large enough and $\theta>0$
such that for $D_{0}$ large enough and for $a_{0}$ small enough,
we have the following: for any $j=1,...,J$, 
\begin{equation}
\mathcal{P}_{j+1}\implies\mathcal{P}_{j}.\label{2eq:-21}
\end{equation}
\end{rem}

\begin{rem}
Note that the inequalities (\ref{2eq:mass_var}) and (\ref{2eq:energ_var})
imply the following inequality for any $t\in [0,t^*]$ in the case
when $P_{j}$ is a breather: 
\begin{align}
\left|\int\left[\widetilde{P_{j}}_{xxxx}+5\widetilde{P_{j}}\widetilde{P_{j}}_{x}^{2}+5\widetilde{P_{j}}^{2}\widetilde{P_{j}}_{xx}+\frac{3}{2}\widetilde{P_{j}}^{5}\right]\varepsilon(t)\right.\nonumber \\
\left.-\int\left[\widetilde{P_{j}}_{xxxx}+5\widetilde{P_{j}}\widetilde{P_{j}}_{x}^{2}+5\widetilde{P_{j}}^{2}\widetilde{P_{j}}_{xx}+\frac{3}{2}\widetilde{P_{j}}^{5}\right]\varepsilon(0)\right| & \leq C\left(\frac{A}{Z_{j}}\right)^{2}\left(a^{2}+e^{-2\theta D}\right),\label{2eq:ineq}
\end{align}
because of the elliptic equation verified by $\widetilde{P_{j}}$,
which, in the case when $\widetilde{P_{j}}=\widetilde{B_{k}}$ is
a breather is the following: 
\begin{equation}
\widetilde{B_{k}}_{xxxx}+5\widetilde{B_{k}}\widetilde{B_{k}}_{x}^{2}+5\widetilde{B_{k}}^{2}\widetilde{B_{k}}_{xx}+\frac{3}{2}\widetilde{B_{k}}^{5}-2\left(\beta_{k}^{2}-\alpha_{k}^{2}\right)\left(\widetilde{B_{k}}_{xx}+\widetilde{B_{k}}^{3}\right)+\left(\alpha_{k}^{2}+\beta_{k}^{2}\right)^{2}\widetilde{B_{k}}=0.\label{2ellip-br}
\end{equation}
But the inequality (\ref{2eq:ineq}) is also true in the case when
$\widetilde{P_{j}}=\widetilde{R_{l}}$ is a soliton, because we have
even better in this case. There are two elliptic equations \cite{key-49}: 
\begin{equation}
\widetilde{R_{l}}_{xxxx}+5\widetilde{R_{l}}\widetilde{R_{l}}_{x}^{2}+5\widetilde{R_{l}}^{2}\widetilde{R_{l}}_{xx}+\frac{3}{2}\widetilde{R_{l}}^{5}-2c_{l}(t)\left(\widetilde{R_{l}}_{xx}+\widetilde{R_{l}}^{3}\right)+c_{l}(t)^{2}\widetilde{R_{l}}=0,
\end{equation}
and 
\begin{equation}
\left(\widetilde{R_{l}}_{xx}+\widetilde{R_{l}}^{3}\right)-c_{l}(t)\widetilde{R_{l}}=0,
\end{equation}
which implies that (\ref{2orth-cond}) implies 
\begin{equation}
\int\left(\widetilde{R_{l}}_{xx}+\widetilde{R_{l}}^{3}\right)\varepsilon=0,\label{2orth-energ}
\end{equation}
and 
\begin{equation}
\int\left(\widetilde{R_{l}}_{xxxx}+5\widetilde{R_{l}}\widetilde{R_{l}}_{x}^{2}+5\widetilde{R_{l}}^{2}\widetilde{R_{l}}_{xx}+\frac{3}{2}\widetilde{R_{l}}^{5}\right)\varepsilon=0,\label{2orth-F}
\end{equation}
which implies of course the inequality (\ref{2eq:ineq}). 
\end{rem}

The proof of Proposition \ref{2prop:ind} will be the goal of the Section
3. The proof of Proposition \ref{2prop:boot} follows from it. 
\begin{proof}[Proof of Proposition \ref{2prop:boot} assuming Proposition \ref{2prop:ind}]
We perform the induction in the decreasing order: $j=J,J-1,...,2,1$.
$\mathcal{P}_{J+1}$ is empty, and Proposition  \ref{2prop:ind} gives the (decreasing) induction step. Hence, $\mathcal{P}_1,\dots,\mathcal{P}_J$ are true. Due to $\mathcal{P}_{1}$: 
\begin{equation}
\forall t\in [0,t^*],\quad\Vert\varepsilon(t)\Vert_{H^{2}}\leq\frac{A}{2}\left(a+e^{-\theta D}\right).
\end{equation}

For $l=1,...,L$, we have from $\mathcal{P}_1$ that for any $t\in [0,t^*]$,
\begin{align}
\vert c_{l}(t)-c_{l}^{0}\vert & \leq\vert c_{l}(t)-c_{l}(0)\vert+\vert c_{l}(0)-c_{l}^{0}\vert\nonumber \\
 & \leq\left(\frac{A}{2}\right)^{2}\left(a^{2}+e^{-2\theta D}\right)+C_{2}a\nonumber \\
 & \leq\left[\left(\frac{A}{2}\right)^{2}\left(a+e^{-\theta D}\right)+C_{2}\right]\left(a+e^{-\theta D}\right).\label{2eq:-20}
\end{align}
If we take $A$ large enough with respect to $C_{2}$, and $a_{0}$
smaller and $D_{0}$ larger if needed with respect to $A$ and $\theta$,
then 
\begin{equation}
\left(\frac{A}{2}\right)^{2}\left(a+e^{-\theta D}\right)+C_{2}\leq\frac{A}{2},
\end{equation}
and that concludes the proof of Proposition \ref{2prop:boot}.

\end{proof}
Hence, we are left to prove Proposition \ref{2prop:ind}. We will write
the proof for a fixed $j\in\{1,...,J\}$. We assume $\mathcal{P}_{j+1}$
with a set of constants $Z_{1}$, ..., $Z_{J+1}$, $A$, $\theta,D_{0}$, $a_{0}$.
We will establish some conditions for these constants during the proof.

\section{Orbital stability of a sum of solitons and breathers in $H^{2}(\mathbb{R})$}
\label{2sec:orb}

In this Section, we prove Proposition \ref{2prop:ind}. We assume $\mathcal{P}_{j+1}$
and we prove $\mathcal{P}_{j}$.

\subsection{Almost decay of conservation laws at the right}

We localize around the most right objects, starting from and including
the $j$-th. We set: 
\begin{equation}
M_{j}(t):=\frac{1}{2}\int u^{2}(t)\Phi_{j}(t)=:M_{j}[u](t),
\end{equation}
\begin{equation}
E_{j}(t):=\int\left[\frac{1}{2}u_{x}^{2}-\frac{1}{4}u^{4}\right]\Phi_{j}(t)=:E_{j}[u](t),
\end{equation}
\begin{equation}
F_{j}(t):=\int\left[\frac{1}{2}u_{xx}^{2}-\frac{5}{2}u^{2}u_{x}^{2}+\frac{1}{4}u^{6}\right]\Phi_{j}(t)=:F_{j}[u](t).
\end{equation}

\begin{lem}
\label{2lem:mono} Let $0<\omega_{1},\omega_{2}<1$. If $0<\sigma<\zeta$,
$0<\theta<\frac{\sqrt{\sigma}}{16}$, $D_{0}$ is large enough and
$a_{0}$ is small enough (depending on $A$, $\theta$, $\omega_{1}$
and $\omega_{2}$), then for any $t\in [0,t^*]$, 
\begin{equation}
M_{j}(t)-M_{j}(0)\leq Ce^{-2\theta D},\label{2mono-m}
\end{equation}
\begin{equation}
\left(E_{j}(t)+\omega_{1}M_{j}(t)\right)-\left(E_{j}(0)+\omega_{1}M_{j}(0)\right)\leq Ce^{-2\theta D},\label{2mono-e}
\end{equation}
\begin{equation}
\left(F_{j}(t)+\omega_{2}M_{j}(t)\right)-\left(F_{j}(0)+\omega_{2}M_{j}(0)\right)\leq Ce^{-2\theta D}.\label{2mono-f}
\end{equation}
\end{lem}

\begin{rem}
\label{2rem:mono}If $j=1$, we have $=0$ at the place of $\leq Ce^{-2\theta D}$,
we will need it in the following of the proof. 
\end{rem}

\begin{proof}[Proof]
If $j=1$, we have exact conservation laws, so this Lemma is obvious.
We assume that $j\geq2$ for the following of this proof. From Appendix
and minoration of $m_{j}'$, 
\begin{align}
\frac{d}{dt}M_{j}(t) & =\int\left[-\frac{3}{2}u_{x}^{2}+\frac{3}{4}u^{4}\right]\Phi_{jx}+\frac{1}{2}\int u^{2}\Phi_{jxxx}+\frac{1}{2}\int u^{2}\Phi_{jt}\nonumber \\
 & =\int\left[-\frac{3}{2}u_{x}^{2}+\frac{3}{4}u^{4}\right]\Phi_{jx}+\frac{1}{2}\int u^{2}\Phi_{jxxx}-\frac{1}{2}m_{j}'\int u^{2}\Phi_{jx}\nonumber \\
 & \leq\int\left[-\frac{3}{2}u_{x}^{2}+\frac{3}{4}u^{4}-\frac{1}{2}\sigma u^{2}\right]\Phi_{jx}+\frac{1}{2}\int u^{2}\Phi_{jxxx}.\label{2eq:-19}
\end{align}

Now, we use that $\left|\Phi_{jxxx}\right|\leq\frac{\sigma}{4}\Phi_{jx}$,
and we obtain that 
\begin{equation}
2\frac{d}{dt}M_{j}(t)\leq-\int\left[-3u_{x}^{2}+\frac{3\sigma}{4}u^{2}-\frac{3}{2}u^{4}\right]\Phi_{jx}.
\end{equation}

Now, from Appendix, we know that for $r>0$, if $t,x$ satisfy $\widetilde{x_{j-1}}(t)+r<x<\widetilde{x_{j}}(t)-r$,
then $\left|\widetilde{P}(t,x)\right|\leq Ce^{-\beta r}$. And so,
for $t,x$ such that $\widetilde{x_{j-1}}(t)+r<x<\widetilde{x_{j}}(t)-r$,
by Sobolev embedding, 
\begin{align}
\left|u(t,x)\right| & \leq\left|\widetilde{P}(t,x)\right|+C\Vert\varepsilon(t)\Vert_{H^{2}}\nonumber \\
 & \leq Ce^{-\beta r}+CA\left(a+e^{-\theta D}\right).\label{2eq:-18}
\end{align}

From that, we can deduce that for $r$ large enough, $a_{0}$ small
enough and $D_{0}$ large enough, for $x\in[\widetilde{x_{j-1}}(t)+r,\widetilde{x_{j}}(t)-r]$,
we can obtain that $\vert u(t,x)\vert$ is bounded by any chosen constant.
Here, we will use that to bound $\frac{3}{2}u^{2}$ by $\frac{\sigma}{4}$.

For $t,x$ such that $x<\widetilde{x_{j-1}}(t)+r$ or $x>\widetilde{x_{j}}(t)-r$:
\begin{align}
\left|\Phi_{jx}(t,x)\right| & \leq C\exp\left(-\frac{\sqrt{\sigma}}{2}\left|x-m_{j}(t)\right|\right)\nonumber \\
 & \leq C\exp\left(-\frac{\sqrt{\sigma}}{2}\left(\frac{D}{2}+\zeta t-r\right)\right)\nonumber \\
 & \leq C\exp\left(-\frac{\sqrt{\sigma}\zeta}{2}t\right)\exp\left(-\frac{\sqrt{\sigma}}{4}D\right)\exp\left(\frac{\sqrt{\sigma}}{2}r\right),\label{2eq:-17}
\end{align}
and so, if we choose $D_{0}$ large enough (more precisely, $D_{0}\geq4r$),
we obtain for $x\notin[\widetilde{x_{j-1}}(t)+r,\widetilde{x_{j}}(t)-r]$:
\begin{equation}
\left|\Phi_{jx}(t,x)\right|\leq C\exp\left(-\frac{\sqrt{\sigma}\zeta}{2}t\right)\exp\left(-\frac{\sqrt{\sigma}}{8}D\right).
\end{equation}

Because $\int u^{4}$ is bounded by a constant for any time (that
depends only on problem data), we deduce that: 
\begin{align}
\frac{d}{dt}M_{j}(t) & \leq-\int\left(\frac{3}{2}u_{x}^{2}+\frac{\sigma}{4}u^{2}\right)\Phi_{jx}(t)+C\exp\left(-\frac{\sqrt{\sigma}\zeta}{2}t\right)\exp\left(-\frac{\sqrt{\sigma}}{8}D\right)\nonumber \\
 & \leq C\exp\left(-\frac{\sqrt{\sigma}\zeta}{2}t\right)\exp\left(-\frac{\sqrt{\sigma}}{8}D\right).\label{2eq:-16}
\end{align}

We deduce what we want to prove by integration.

For the second inequality, the argument is similar. We start by
using Appendix: 
\begin{align}
\frac{d}{dt}E_{j}(t) & =\int\left[-\frac{1}{2}\left(u_{xx}+u^{3}\right)^{2}-u_{xx}^{2}+3u^{2}u_{x}^{2}\right]\Phi_{jx}+\frac{1}{2}\int u_{x}^{2}\Phi_{jxxx}-m_{j}'\int\left[\frac{1}{2}u_{x}^{2}-\frac{1}{4}u^{4}\right]\Phi_{jx}\nonumber \\
 & \leq-\int\left[\frac{1}{2}\left(u_{xx}+u^{3}\right)^{2}+u_{xx}^{2}-3u^{2}u_{x}^{2}+\frac{\sigma}{2}u_{x}^{2}-\frac{v_{J}}{2}u^{4}\right]\Phi_{jx}+\frac{1}{2}\int u_{x}^{2}\Phi_{jxxx},\label{2eq:-15}
\end{align}
and again, because $\left|\Phi_{jxxx}\right|\leq\frac{\sigma}{4}\Phi_{jx}$,
we obtain that 
\begin{equation}
2\frac{d}{dt}E_{j}(t)\leq-\int\left[\left(u_{xx}+u^{3}\right)^{2}+2u_{xx}^{2}-6u^{2}u_{x}^{2}+\frac{3\sigma}{4}u_{x}^{2}-v_{J}u^{4}\right]\Phi_{jx},
\end{equation}
and by doing a similar argument as for the mass, but to bound $6u^{2}$
by $\frac{\sigma}{4}$ and to bound $v_{J}u^{2}$ by $\omega_{3}>0$,
a constant as small as we need, we obtain 
\begin{align}
\frac{d}{dt}E_{j}(t) & \leq-\int\left[\frac{1}{2}\left(u_{xx}+u^{3}\right)^{2}+u_{xx}^{2}+\frac{\sigma}{4}u_{x}^{2}-\frac{\omega_{3}}{2}u^{2}\right]\Phi_{jx}+C\exp\left(-\frac{\sqrt{\sigma}\zeta}{2}t\right)\exp\left(-\frac{\sqrt{\sigma}}{8}D\right)\nonumber \\
 & \leq\frac{\omega_{3}}{2}\int u^{2}\Phi_{jx}+C\exp\left(-\frac{\sqrt{\sigma}\zeta}{2}t\right)\exp\left(-\frac{\sqrt{\sigma}}{8}D\right),\label{2eq:-14}
\end{align}
and so, if we choose $\omega_{3}$ such that $\frac{\omega_{3}}{2}\leq\omega_{1}\frac{\sigma}{4}$,
\begin{align}
\frac{d}{dt}\left(E_{j}(t)+\omega_{1}M_{j}(t)\right) & \leq\frac{\omega_{3}}{2}\int u^{2}\Phi_{jx}+C\exp\left(-\frac{\sqrt{\sigma}\zeta}{2}t\right)\exp\left(-\frac{\sqrt{\sigma}}{8}D\right)-\frac{\omega_{1}\sigma}{4}\int u^{2}\Phi_{jx}\nonumber \\
 & \leq C\exp\left(-\frac{\sqrt{\sigma}\zeta}{2}t\right)\exp\left(-\frac{\sqrt{\sigma}}{8}D\right),\label{2eq:-13}
\end{align}
and we deduce what we want to prove by integration.

For the third inequality, the argument is similar. We start by using
Appendix: 
\begin{align}
\frac{d}{dt}F_{j}(t) & =\int\left(-\frac{3}{2}u_{xxx}^{2}+9u_{xx}^{2}u^{2}+15u_{x}^{2}uu_{xx}+\frac{9}{16}u^{8}+\frac{1}{4}u_{x}^{4}+\frac{3}{2}u_{xx}u^{5}-\frac{45}{4}u^{4}u_{x}^{2}\right)\Phi_{jx}\nonumber \\
 & +5\int u^{2}u_{x}u_{xx}\Phi_{jxx}+\frac{1}{2}\int u_{xx}^{2}\Phi_{jxxx}-m_{j}'\int\left[\frac{1}{2}u_{xx}^{2}-\frac{5}{2}u^{2}u_{x}^{2}+\frac{1}{4}u^{6}\right]\Phi_{jx}\nonumber \\
 & \leq-\int\left(\frac{3}{2}u_{xxx}^{2}-9u_{xx}^{2}u^{2}-15u_{x}^{2}uu_{xx}-\frac{9}{16}u^{8}-\frac{1}{4}u_{x}^{4}-\frac{3}{2}u_{xx}u^{5}+\frac{45}{4}u^{4}u_{x}^{2}+\frac{\sigma}{2}u_{xx}^{2}\right.   \nonumber \\
 & \left. -5v_{J}u^{2}u_{x}^{2}+\frac{\sigma}{4}u^{6}\right)\Phi_{jx} +5\int u^{2}u_{x}u_{xx}\Phi_{jxx}+\frac{1}{2}\int u_{xx}^{2}\Phi_{jxxx},\label{2eq:-12}
\end{align}
and again, because $\left|\Phi_{jxx}\right|\leq\frac{\sqrt{\sigma}}{2}\Phi_{jx}$
and $\left|\Phi_{jxxx}\right|\leq\frac{\sigma}{4}\Phi_{jx}$, we obtain
that 
\begin{align}
2\frac{d}{dt}F_{j}(t) & \leq-\int\left(3u_{xxx}^{2}-\left(18+\frac{5}{2}\sqrt{\sigma}\right)u_{xx}^{2}u^{2}-30u_{x}^{2}uu_{xx}-\frac{9}{8}u^{8}-\frac{1}{2}u_{x}^{4}-3u_{xx}u^{5}\right.\nonumber \\
 & \left.+\frac{45}{2}u^{4}u_{x}^{2}+\frac{3\sigma}{4}u_{xx}^{2}-\left(10v_{J}+\frac{5}{2}\sqrt{\sigma}\right)u^{2}u_{x}^{2}+\frac{\sigma}{2}u^{6}\right)\Phi_{jx},\label{2eq:-11}
\end{align}
and by doing a similar argument as for the mass, but to bound $\left(18+\frac{5}{2}\sqrt{\sigma}\right)u^{2}$
by $\frac{\sigma}{4}$, to bound $30uu_{xx}$ by $\omega_{4}$, to
bound $\frac{9}{8}u^{6}$ by $\omega_{5}$, to bound $\frac{1}{2}u_{x}^{2}$
by $\omega_{4}$, to bound $3u_{xx}u^{3}$ by $\omega_{5}$ and to
bound $\left(10v_{J}+\frac{5}{2}\sqrt{\sigma}\right)u^{2}$ by $\omega_{4}$,
where $\omega_{3},\omega_{4}>0$ are constants that we can choose
as small as we want. And we obtain 
\begin{align}
\frac{d}{dt}F_{j}(t) & \leq-\int\left[\frac{3}{2}u_{xxx}^{2}+\frac{45}{4}u^{4}u_{x}^{2}+\frac{\sigma}{4}u^{6}+\frac{\sigma}{4}u_{xx}^{2}-C\omega_{4}u_{x}^{2}-C\omega_{5}u^{2}\right]\Phi_{jx}\nonumber\\
& +C\exp\left(-\frac{\sqrt{\sigma}\zeta}{2}t\right)\exp\left(-\frac{\sqrt{\sigma}}{8}D\right)\nonumber \\
 & \leq C\omega_{4}\int u_{x}^{2}\Phi_{jx}+C\omega_{5}\int u^{2}\Phi_{jx}+C\exp\left(-\frac{\sqrt{\sigma}\zeta}{2}t\right)\exp\left(-\frac{\sqrt{\sigma}}{8}D\right),\label{2eq:-10}
\end{align}
and so, if we choose $\omega_{4},\omega_{5}$ such that $C\omega_{4}\leq\omega_{1}\frac{3}{2}$
and $C\omega_{5}\leq\omega_{1}\frac{\sigma}{4}$, 
\begin{equation}
\frac{d}{dt}\left(F_{j}(t)+\omega_{2}M_{j}(t)\right)\leq C\exp\left(-\frac{\sqrt{\sigma}\zeta}{2}t\right)\exp\left(-\frac{\sqrt{\sigma}}{8}D\right),
\end{equation}
and we deduce what we want to prove by integration. 
\end{proof}

\subsection{Quadratic approximation for conservation laws at the right}

We can write the following Taylor expansions with $u=\widetilde{P}+\varepsilon$
for any $t\in [0,t^*]$:

\begin{equation}
M_{j}(t)-M_{j}[\widetilde{P}](t)-\int\widetilde{P}\varepsilon\Phi_{j}-\frac{1}{2}\int\varepsilon^{2}\Phi_{j}=0,\label{2mass-quadr}
\end{equation}
\begin{equation}
\left|E_{j}(t)-E_{j}[\widetilde{P}](t)-\int\left[\widetilde{P}_{x}\varepsilon_{x}-\widetilde{P}^{3}\varepsilon\right]\Phi_{j}-\int\left[\frac{1}{2}\varepsilon_{x}^{2}-\frac{3}{2}\widetilde{P}^{2}\varepsilon^{2}\right]\Phi_{j}\right|\leq CA\left(a+e^{-\theta D}\right)\int\varepsilon^{2}\Phi_{j},\label{2energ-quadr}
\end{equation}
\begin{align}
\left|F_{j}(t)-F_{j}[\widetilde{P}](t)-\int\left[\widetilde{P}_{xx}\varepsilon_{xx}-5\widetilde{P}\widetilde{P}_{x}^{2}\varepsilon-5\widetilde{P}^{2}\widetilde{P}_{x}\varepsilon_{x}+\frac{3}{2}\widetilde{P}^{5}\varepsilon\right]\Phi_{j}\right.\label{2f-quadr}\\
\left.-\int\left[\frac{1}{2}\varepsilon_{xx}^{2}-\frac{5}{2}\widetilde{P}_{x}^{2}\varepsilon^{2}-10\widetilde{P}\widetilde{P}_{x}\varepsilon\varepsilon_{x}-\frac{5}{2}\widetilde{P}^{2}\varepsilon_{x}^{2}+\frac{15}{4}\widetilde{P}^{4}\varepsilon^{2}\right]\Phi_{j}\right| & \leq CA\left(a+e^{-\theta D}\right)\int\left(\varepsilon^{2}+\varepsilon_{x}^{2}\right)\Phi_{j}.\nonumber 
\end{align}

Now, we want to simplify each term of the Taylor expansion.

\subsubsection{Constant terms of the Taylor expansion}

We obtain the following lemma dealing with variations of the constant
parts of each Taylor expansion of conservation laws at the right.
We reduce each variation to the variation of each conservation law
of $\widetilde{P_{j}}$. Note that if $P_{j}$ is a breather, the
variation of any conservation law of $\widetilde{P_{j}}$ is $0$.
But, if $P_{j}=R_{l}$, we may express in the following way $M[\widetilde{R_{l}}]$,
$E[\widetilde{R_{l}}]$ and $F[\widetilde{R_{l}}]$ with respect to
$Q$, the ground state of parameter $c=1$ (the basic ground state):

\begin{equation}
M[\widetilde{R_{l}}](t)=c_{l}(t)^{1/2}M[Q],\label{2solit-mass}
\end{equation}
\begin{equation}
E[\widetilde{R_{l}}](t)=c_{l}(t)^{3/2}E[Q],\label{2solit-energ}
\end{equation}
\begin{equation}
F[\widetilde{R_{l}}](t)=c_{l}(t)^{5/2}F[Q].
\end{equation}

\begin{lem}
\label{2lem:const}If $\sigma<4\beta^{2},\theta<\min\left(\frac{\beta}{4},\frac{\sqrt{\sigma}}{8}\right)$,
then for any $t\in [0,t^*]$, 
\begin{equation}
\left|M_{j}[\widetilde{P}](t)-M_{j}[\widetilde{P}](0)-\left(M[\widetilde{P_{j}}](t)-M[\widetilde{P_{j}}](0)\right)\right|\leq Ce^{-2\theta D}+C\left(\frac{A}{Z_{j+1}}\right)^{2}\left(a^{2}+e^{-2\theta D}\right),\label{2var-mass}
\end{equation}
\begin{equation}
\left|E_{j}[\widetilde{P}](t)-E_{j}[\widetilde{P}](0)-\left(E[\widetilde{P_{j}}](t)-E[\widetilde{P_{j}}](0)\right)\right|\leq Ce^{-2\theta D}+C\left(\frac{A}{Z_{j+1}}\right)^{2}\left(a^{2}+e^{-2\theta D}\right),\label{2var-energ}
\end{equation}
\begin{equation}
\left|F_{j}[\widetilde{P}](t)-F_{j}[\widetilde{P}](0)-\left(F[\widetilde{P_{j}}](t)-F[\widetilde{P_{j}}](0)\right)\right|\leq Ce^{-2\theta D}+C\left(\frac{A}{Z_{j+1}}\right)^{2}\left(a^{2}+e^{-2\theta D}\right).
\end{equation}
\end{lem}

\begin{rem}
\label{2rem:const}In the case when $P_{j}$ is a breather, we have
that
\begin{equation}
M[\widetilde{P_{j}}](t)-M[\widetilde{P_{j}}](0)=E[\widetilde{P_{j}}](t)-E[\widetilde{P_{j}}](0)=F[\widetilde{P_{j}}](t)-F[\widetilde{P_{j}}](0)=0.\label{2eq:-9}
\end{equation}

It is not true in the case when $P_{j}$ is a soliton.
\end{rem}

\begin{proof}[Proof]
When we develop using $\widetilde{P}=\sum_{i=1}^{J}\widetilde{P_{i}}$,
we obtain terms with $\widetilde{P_{i}}\widetilde{P_{j}}$ with $i\neq j$
and the other terms that have all the same index. For the first type
of terms, it is enough to bound $\int\widetilde{P_{i}}\widetilde{P_{j}}$
for $i\neq j$: 
\begin{equation}
\left|\int\widetilde{P_{i}}\widetilde{P_{j}}\right|\leq Ce^{-\frac{\beta D}{2}}.
\end{equation}

Now, we look on the terms with the same index, for example $\int\widetilde{P_{i}}^{2}\Phi_{j}$.
We will distinguish several cases. If $i<j$, 
\begin{align}
\int\widetilde{P_{i}}^{2}\Phi_{j} & \leq C\int e^{-2\beta\left|x-\widetilde{x_{i}}(t)\right|}e^{\frac{\sqrt{\sigma}}{2}\left(x-m_{j}(t)\right)}dx\nonumber \\
 & =C\int_{-\infty}^{\widetilde{x_{i}}(t)}e^{\left(2\beta+\sqrt{\sigma}/2\right)x-2\beta\widetilde{x_{i}}(t)-\frac{\sqrt{\sigma}}{2}m_{j}(t)}dx+C\int_{\widetilde{x_{i}}(t)}^{+\infty}e^{\left(-2\beta+\sqrt{\sigma}/2\right)x+2\beta\widetilde{x_{i}}(t)-\frac{\sqrt{\sigma}}{2}m_{j}(t)}dx\nonumber \\
 & \leq\frac{C}{\sqrt{\sigma}}e^{\frac{\sqrt{\sigma}}{2}\left(\widetilde{x_{i}}(t)-m_{j}(t)\right)}\nonumber \\
 & \leq Ce^{-\frac{\sqrt{\sigma}D}{4}}.\label{2eq:-8}
\end{align}

For the same reason and properties of $\Psi$, for $i\geq j$, 
\begin{equation}
\int\widetilde{P_{i}}^{2}\left(1-\Phi_{j}\right)\leq Ce^{-\frac{\sqrt{\sigma}D}{4}}.
\end{equation}

For $i>j$, we may use $\mathcal{P}_{i}$; and for $i=j$, we cannot.
So, if for $i\geq j+1$, $\widetilde{P_{i}}=\widetilde{R_{l}}$ is
a soliton, we have by the mean-value theorem, using (\ref{2velo-assump}),
\begin{align}
\left|M[\widetilde{P_{i}}](t)-M[\widetilde{P_{i}}](0)\right| & =\left|c_{l}(t)^{1/2}-c_{l}(0)^{1/2}\right|\left|M[Q]\right|\nonumber \\
 & \leq C\left|c_{l}(t)-c_{l}(0)\right|\nonumber \\
 & \leq C\left(\frac{A}{Z_{j+1}}\right)^{2}\left(a^{2}+e^{-2\theta D}\right),\label{2eq:-7}
\end{align}
and, by the same way, we may obtain, 
\begin{equation}
\left|E[\widetilde{P_{i}}](t)-E[\widetilde{P_{i}}](0)\right|\leq C\left(\frac{A}{Z_{j+1}}\right)^{2}\left(a^{2}+e^{-2\theta D}\right),
\end{equation}
and 
\begin{equation}
\left|F[\widetilde{P_{i}}](t)-F[\widetilde{P_{i}}](0)\right|\leq C\left(\frac{A}{Z_{j+1}}\right)^{2}\left(a^{2}+e^{-2\theta D}\right).
\end{equation}
\end{proof}

\subsubsection{Linear terms of the Taylor expansion}

We denote: 
\begin{equation}
m[X]:=\int X\varepsilon,
\end{equation}
\begin{equation}
m_{j}[X]:=\int X\varepsilon\Phi_{j},
\end{equation}
\begin{equation}
e[X]:=\int\left[X_{x}\varepsilon_{x}-X^{3}\varepsilon\right]=-\int\left[X_{xx}+X^{3}\right]\varepsilon,
\end{equation}
\begin{equation}
e_{j}[X]:=\int\left[X_{x}\varepsilon_{x}-X^{3}\varepsilon\right]\Phi_{j},
\end{equation}
\begin{equation}
f[X]:=\int\left[X_{xx}\varepsilon_{xx}-5XX_{x}^{2}\varepsilon-5X^{2}X_{x}\varepsilon_{x}+\frac{3}{2}X^{5}\varepsilon\right]=\int\left[X_{xxxx}+5XX_{x}^{2}+5X^{2}X_{xx}+\frac{3}{2}X^{5}\right]\varepsilon,
\end{equation}
\begin{equation}
f_{j}[X]:=\int\left[X_{xx}\varepsilon_{xx}-5XX_{x}^{2}\varepsilon-5X^{2}X_{x}\varepsilon_{x}+\frac{3}{2}X^{5}\varepsilon\right]\Phi_{j}.
\end{equation}

\begin{lem}
\label{2lem:lin}If $\sigma<\beta^{2},\theta<\min\left(\frac{\beta}{4},\frac{\sqrt{\sigma}}{8}\right)$,
if $D_{0}$ is large enough and $a_{0}$ is small enough, then for
any $t\in [0,t^*]$, 
\begin{equation}
\left|m_{j}[\widetilde{P}](t)-m_{j}[\widetilde{P}](0)-\left(m[\widetilde{P_{j}}](t)-m[\widetilde{P_{j}}](0)\right)\right|\leq Ce^{-2\theta D}+C\left(\frac{A}{Z_{j+1}}\right)^{2}\left(a^{2}+e^{-2\theta D}\right),\label{2varlinmass}
\end{equation}
\begin{equation}
\left|e_{j}[\widetilde{P}](t)-e_{j}[\widetilde{P}](0)-\left(e[\widetilde{P_{j}}](t)-e[\widetilde{P_{j}}](0)\right)\right|\leq Ce^{-2\theta D}+C\left(\frac{A}{Z_{j+1}}\right)^{2}\left(a^{2}+e^{-2\theta D}\right),\label{2varlinenerg}
\end{equation}
\begin{equation}
\left|f_{j}[\widetilde{P}](t)-f_{j}[\widetilde{P}](0)-\left(f[\widetilde{P_{j}}](t)-f[\widetilde{P_{j}}](0)\right)\right|\leq Ce^{-2\theta D}+C\left(\frac{A}{Z_{j+1}}\right)^{2}\left(a^{2}+e^{-2\theta D}\right).
\end{equation}
\end{lem}

\begin{proof}[Proof]
We develop using $\widetilde{P}=\sum_{i=1}^{J}\widetilde{P_{i}}$.
We obtain terms with $\widetilde{P_{i}}\widetilde{P_{j}}$ with $i\neq j$
and the other terms that have all the same index. Knowing that $\varepsilon$
bounded for $D_{0}$ large enough and $a_{0}$ small enough (with
respect to $A$), we obtain the same bounds in the same way as for
the constant part.

Now, if, for $i\geq j+1$, $\widetilde{P_{i}}$ is a soliton, then
we have simply: $m[\widetilde{P_{i}}]=e[\widetilde{P_{i}}]=f[\widetilde{P_{i}}]=0.$

If, for $i\geq j+1$, $\widetilde{P_{i}}$ is a breather, we have
a bound for the variation of these quantities by $\mathcal{P}_{i}$. 
\end{proof}

\subsubsection{Quadratic part of the Taylor expansion}

We set: 
\begin{equation}
\mathcal{M}_{j}[X]:=\frac{1}{2}\int\varepsilon^{2}\Phi_{j},
\end{equation}
\begin{equation}
\mathcal{E}_{j}[X]:=\int\left[\frac{1}{2}\varepsilon_{x}^{2}-\frac{3}{2}X^{2}\varepsilon^{2}\right]\Phi_{j},
\end{equation}
\begin{equation}
\mathcal{F}_{j}[X]:=\int\left[\frac{1}{2}\varepsilon_{xx}^{2}-\frac{5}{2}X_{x}^{2}\varepsilon^{2}-10XX_{x}\varepsilon\varepsilon_{x}-\frac{5}{2}X^{2}\varepsilon_{x}^{2}+\frac{15}{4}X^{4}\varepsilon^{2}\right]\Phi_{j},
\end{equation}
and $\mathcal{M}_{j}(t):=\mathcal{M}_{j}[\widetilde{P}](t)$, ... 
\begin{lem}
\label{2lem:quad}If $\sigma<\beta^{2},\theta<\frac{\sqrt{\sigma}}{8}$,
if $D_{0}$ is large enough and $a_{0}$ is small enough, then for
any $t\in [0,t^*]$, 
\begin{equation}
\left|\mathcal{M}_{j}[\widetilde{P}](t)-\mathcal{M}_{j}[\widetilde{P_{j}}](t)\right|\leq Ce^{-2\theta D}+C\left(\frac{A}{Z_{j+1}}\right)^{2}\left(a^{2}+e^{-2\theta D}\right),\label{2quad-mass}
\end{equation}
\begin{equation}
\left|\mathcal{E}_{j}[\widetilde{P}](t)-\mathcal{E}_{j}[\widetilde{P_{j}}](t)\right|\leq Ce^{-2\theta D}+C\left(\frac{A}{Z_{j+1}}\right)^{2}\left(a^{2}+e^{-2\theta D}\right),
\end{equation}
\begin{equation}
\left|\mathcal{F}_{j}[\widetilde{P}](t)-\mathcal{F}_{j}[\widetilde{P_{j}}](t)\right|\leq Ce^{-2\theta D}+C\left(\frac{A}{Z_{j+1}}\right)^{2}\left(a^{2}+e^{-2\theta D}\right).
\end{equation}
\end{lem}

\begin{proof}[Proof]
We develop using $\widetilde{P}=\sum_{i=1}^{J}\widetilde{P_{i}}$.
For terms with $\widetilde{P_{i}}$, with $i>j$, we use the induction
assumption for $\varepsilon$. For terms with $\widetilde{P_{i}}$,
with $i<j$, we do as in the previous sections. 
\end{proof}
Note that (\ref{2quad-mass}) is useless, because $\mathcal{M}_{j}[\widetilde{P}](t)-\mathcal{M}_{j}[\widetilde{P_{j}}](t)=0$
since $\mathcal{M}_{j}[X]$ do not depend on $X$, but we write it
in order to argue in the same way for the three conserved quantities.

\subsection{Lyapunov functional and simplifications}

We introduce the following Lyapunov functional: 
\begin{equation}
\mathcal{H}_{j}(t):=F_{j}(t)+2\left(b_{j}^{2}-a_{j}^{2}\right)E_{j}(t)+\left(a_{j}^{2}+b_{j}^{2}\right)^{2}M_{j}(t).
\end{equation}

We set:
\begin{equation}
\mathcal{K}(t):=F[\widetilde{P_{j}}](t)+2\left(b_{j}^{2}-a_{j}^{2}\right)E[\widetilde{P_{j}}](t)+\left(a_{j}^{2}+b_{j}^{2}\right)^{2}M[\widetilde{P_{j}}](t),
\end{equation}
\begin{equation}
\mathcal{L}(t):=f[\widetilde{P_{j}}](t)+2\left(b_{j}^{2}-a_{j}^{2}\right)e[\widetilde{P_{j}}](t)+\left(a_{j}^{2}+b_{j}^{2}\right)^{2}m[\widetilde{P_{j}}](t),
\end{equation}
\begin{equation}
\mathcal{Q}(t):=\mathcal{F}_{j}[\widetilde{P_{j}}](t)+2\left(b_{j}^{2}-a_{j}^{2}\right)\mathcal{E}_{j}[\widetilde{P_{j}}](t)+\left(a_{j}^{2}+b_{j}^{2}\right)^{2}\mathcal{M}_{j}[\widetilde{P_{j}}](t).
\end{equation}

We have the following:
\begin{lem}
If $\sigma<\beta^{2},\theta<\min\left(\frac{\beta}{4},\frac{\sqrt{\sigma}}{8}\right)$,
if $D_{0}$ is large enough and $a_{0}$ is small enough, then for
any $t\in [0,t^*]$, 
\begin{equation}
\left|\mathcal{H}_{j}(t)-\mathcal{H}_{j}(0)-\left(\mathcal{Q}(t)-\mathcal{Q}(0)\right)\right|\leq Ce^{-2\theta D}+C\left(\frac{A}{Z_{j+1}}\right)^{2}\left(a^{2}+e^{-2\theta D}\right).\label{2quad-var}
\end{equation}
\end{lem}

\begin{proof}[Proof]
From (\ref{2mass-quadr}), (\ref{2energ-quadr}), (\ref{2f-quadr}),
Lemmas \ref{2lem:const}, \ref{2lem:lin} and \ref{2lem:quad} we deduce
that:

if $\sigma<\beta^{2}$, $\theta>\min\left(\frac{\beta}{4},\frac{\sqrt{\sigma}}{8}\right)$,
if $D_{0}$ is large enough and $a_{0}$ is small enough, then for
any $t\in [0,t^*]$,
\begin{align}
\mathcal{H}_{j}(t)-\mathcal{H}_{j}(0) & =\mathcal{K}(t)-\mathcal{K}(0)+\mathcal{L}(t)-\mathcal{L}(0)+\mathcal{Q}(t)-\mathcal{Q}(0)+O\left(A\left(a+e^{-\theta D}\right)\int\left(\varepsilon^{2}+\varepsilon_{x}^{2}\right)\Phi_{j}\right)\label{2eq:apprH}\nonumber\\
 & +O\left(e^{-2\theta D}\right)+O\left(\left(\frac{A}{Z_{j+1}}\right)^{2}\left(a^{2}+e^{-2\theta D}\right)\right).
\end{align}

Now, from the bootstrap assumption (\ref{2cond1}), we see that if
we take $D_{0}$ large enough and $a_{0}$ small enough, then we can
bound $A\left(a+e^{-\theta D}\right)\int\left(\varepsilon^{2}+\varepsilon_{x}^{2}\right)\Phi_{j}$
by $\left(\frac{A}{Z_{j+1}}\right)^{2}\left(a^{2}+e^{-2\theta D}\right)$.
So, for any $t\in [0,t^*]$,
\begin{equation}\begin{aligned}
\MoveEqLeft\mathcal{H}_{j}(t)-\mathcal{H}_{j}(0)\\
&=\mathcal{K}(t)-\mathcal{K}(0)+\mathcal{L}(t)-\mathcal{L}(0)+\mathcal{Q}(t)-\mathcal{Q}(0)+O\left(e^{-2\theta D}\right)+O\left(\left(\frac{A}{Z_{j+1}}\right)^{2}\left(a^{2}+e^{-2\theta D}\right)\right).\label{2eq:}
\end{aligned}\end{equation}

Now, we will simplify $\mathcal{K}(t)-\mathcal{K}(0)$ and $\mathcal{L}(t)-\mathcal{L}(0)$.

\textbf{Simplification of $\mathcal{K}(t)-\mathcal{K}(0)$:}

If $\widetilde{P_{j}}$ is a breather, then $\mathcal{K}(t)-\mathcal{K}(0)=0$.
If $\widetilde{P_{j}}=\widetilde{R_{l}}$ is a soliton, then we have
\begin{align}
\mathcal{K}(t) & =F[\widetilde{R_{l}}](t)+2c_{l}^{0}E[\widetilde{R_{l}}](t)+\left(c_{l}^{0}\right)^{2}M[\widetilde{R_{l}}](t)\nonumber \\
 & =c_{l}(t)^{5/2}F[Q]+2c_{l}^{0}c_{l}(t)^{3/2}E[Q]+\left(c_{l}^{0}\right)^{2}c_{l}(t)^{1/2}M[Q].\label{2eq:-6}
\end{align}

Now, we observe that $c_{l}(t)=c_{l}^{0}+\left(c_{l}(t)-c_{l}^{0}\right)$,
this is why we want to do a Taylor expansion for each power function.
We recall that by (\ref{2eq:second}), for $D_{0}$ large enough and
$a_{0}$ small enough (with respect to $Z_{j+1}$), $\left|c_{l}(t)-c_{l}^{0}\right|^{3}$
is bounded by $C\left(\frac{A}{Z_{j+1}}\right)^{2}\left(a^{2}+e^{-2\theta D}\right)$.
This is why we may approximate $\mathcal{K}(t)$ by a Taylor expansion
of order 2: 
\begin{align}
\left(c_{l}^{0}\right)^{5/2}\left(1+\frac{5}{2}\frac{c_{l}(t)-c_{l}^{0}}{c_{l}^{0}}+\frac{15}{8}\left(\frac{c_{l}(t)-c_{l}^{0}}{c_{l}^{0}}\right)^{2}\right)F[Q]\notag\\
+2\left(c_{l}^{0}\right)^{5/2}\left(1+\frac{3}{2}\frac{c_{l}(t)-c_{l}^{0}}{c_{l}^{0}}+\frac{3}{8}\left(\frac{c_{l}(t)-c_{l}^{0}}{c_{l}^{0}}\right)^{2}\right)E[Q]\nonumber \\
+\left(c_{l}^{0}\right)^{5/2}\left(1+\frac{1}{2}\frac{c_{l}(t)-c_{l}^{0}}{c_{l}^{0}}-\frac{1}{8}\left(\frac{c_{l}(t)-c_{l}^{0}}{c_{l}^{0}}\right)^{2}\right)M[Q]\label{2eq:-5}\\
=\left(c_{l}^{0}\right)^{5/2}\left(F[Q]+2E[Q]+M[Q]\right)+\left(c_{l}(t)-c_{l}^{0}\right)\left(c_{l}^{0}\right)^{3/2}\left(\frac{5}{2}F[Q]+3E[Q]+\frac{1}{2}M[Q]\right)\nonumber \\
+\left(c_{l}(t)-c_{l}^{0}\right)^{2}\left(c_{l}^{0}\right)^{1/2}\left(\frac{15}{8}F[Q]+\frac{3}{4}E[Q]-\frac{1}{8}M[Q]\right).\nonumber 
\end{align}

Now, we use the fact that $M[Q]=2$, $E[Q]=-\frac{2}{3}$ and $F[Q]=\frac{2}{5}$,
and we obtain that the Taylor expression of order 2 is in fact: 
\begin{equation}
\frac{16}{15}\left(c_{l}^{0}\right)^{5/2}+0\left(c_{l}(t)-c_{l}^{0}\right)\left(c_{l}^{0}\right)^{3/2}+0\left(c_{l}(t)-c_{l}^{0}\right)\left(c_{l}^{0}\right)^{3/2}=\frac{16}{15}\left(c_{l}^{0}\right)^{5/2}.
\end{equation}

Thus, 
\begin{equation}
\left|\mathcal{K}(t)-\frac{16}{15}\left(c_{l}^{0}\right)^{5/2}\right|\leq C\left(\frac{A}{Z_{j+1}}\right)^{2}\left(a^{2}+e^{-2\theta D}\right),
\end{equation}
and so, 
\begin{equation}
\left|\mathcal{K}(t)-\mathcal{K}(0)\right|\leq C\left(\frac{A}{Z_{j+1}}\right)^{2}\left(a^{2}+e^{-2\theta D}\right).
\end{equation}

\textbf{Simplification of $\mathcal{L}(t)-\mathcal{L}(0)$:}

If $\widetilde{P_{j}}$ is a breather, we have, by the elliptic equation
verified by a breather, that $\mathcal{L}(t)=0$.

If $\widetilde{P_{j}}$ is a soliton, we have, by (\ref{2orth-cond}),
(\ref{2orth-energ}) and (\ref{2orth-F}), that $\mathcal{L}(t)=0$
(we have simply $m[\widetilde{P_{j}}]=e[\widetilde{P_{j}}]=f[\widetilde{P_{j}}]=0$).
\end{proof}

\subsection{Coercivity}
\begin{lem}
If $\sigma$ is small enough (with respect to constants that depend
only on problem data), \begin{align}
\theta<\min\left(\frac{\beta}{4},\frac{\sqrt{\sigma}}{8}\right),
\end{align} 
if $D_{0}$ is large enough and $a_{0}$ is small enough, then for
any $t\in [0,t^*]$, 
\begin{equation}
\int\left(\varepsilon^{2}+\varepsilon_{x}^{2}+\varepsilon_{xx}^{2}\right)\Phi_{j}\leq C\mathcal{Q}(t)+C\left(\int\varepsilon\widetilde{P_{j}}\right)^{2}+Ce^{-2\theta D}.\label{2coercivity}
\end{equation}
\end{lem}

\begin{proof}[Proof]
We notice that $\int\left(\varepsilon^{2}+\varepsilon_{x}^{2}+\varepsilon_{xx}^{2}\right)\Phi_{j}$
is $\Vert\varepsilon\sqrt{\Phi_{j}}\Vert_{H^{2}}^{2}$ modulo some
terms that can be bounded by \begin{align}
\label{2modulo}
C\sqrt{\sigma}\int\left(\varepsilon^{2}+\varepsilon_{x}^{2}+\varepsilon_{xx}^{2}\right)\Phi_{j}.
\end{align} 
This is why we will bound $\Vert\varepsilon\sqrt{\Phi_{j}}\Vert_{H^{2}}^{2}$.
We can bound it by the canonical quadratic form associated to $\widetilde{P_{j}}$
and evaluated in $\varepsilon\sqrt{\Phi_{j}}$, if $\varepsilon\sqrt{\Phi_{j}}$
satisfies quite well the orthogonality conditions, which is the case
(see Sections 5.4 and (372) in \cite{key-49}).

So, we obtain that, if $\widetilde{P_{j}}$ is a breather, the canonical
quadratic form is $\mathcal{Q}(t)$ modulo \eqref{2modulo},
and we have 
\begin{equation}
\int\left(\varepsilon^{2}+\varepsilon_{x}^{2}+\varepsilon_{xx}^{2}\right)\Phi_{j}\leq C\mathcal{Q}(t)+C\left(\int\varepsilon\sqrt{\Phi_{j}}\widetilde{P_{j}}\right)^{2}+C\sqrt{\sigma}\int\left(\varepsilon^{2}+\varepsilon_{x}^{2}+\varepsilon_{xx}^{2}\right)\Phi_{j},
\end{equation}
which means that if $\sigma$ is small enough, 
\begin{equation}
\int\left(\varepsilon^{2}+\varepsilon_{x}^{2}+\varepsilon_{xx}^{2}\right)\Phi_{j}\leq C\mathcal{Q}(t)+C\left(\int\varepsilon\sqrt{\Phi_{j}}\widetilde{P_{j}}\right)^{2},
\end{equation}
and we check that $\int\varepsilon\sqrt{\Phi_{j}}\widetilde{P_{j}}$
is $\int\varepsilon\widetilde{P_{j}}$ modulo $Ce^{-\theta D}$. So,
\begin{equation}
\int\left(\varepsilon^{2}+\varepsilon_{x}^{2}+\varepsilon_{xx}^{2}\right)\Phi_{j}\leq C\mathcal{Q}(t)+C\left(\int\varepsilon\widetilde{P_{j}}\right)^{2}+Ce^{-2\theta D}.
\end{equation}

If $\widetilde{P_{j}}=\widetilde{R_{l}}$ is a soliton, the canonical
quadratic form is, modulo $C\sqrt{\sigma}\int\left(\varepsilon^{2}+\varepsilon_{x}^{2}+\varepsilon_{xx}^{2}\right)\Phi_{j}$,
\begin{equation}
\mathcal{Q}_{0}(t):=\mathcal{F}_{j}[\widetilde{R_{l}}](t)+2c_{l}(t)\mathcal{E}_{j}[\widetilde{R_{l}}](t)+c_{l}(t)^{2}\mathcal{M}_{j}[\widetilde{R_{l}}](t).
\end{equation}

This is why, for the same reasons as above, we have 
\begin{equation}
\int\left(\varepsilon^{2}+\varepsilon_{x}^{2}+\varepsilon_{xx}^{2}\right)\Phi_{j}\leq C\mathcal{Q}_{0}(t).
\end{equation}

This is why, we need to be able to bound $\left|\mathcal{Q}_{0}(t)-\mathcal{Q}(t)\right|$.
We have 
\begin{equation}
\mathcal{Q}_{0}(t)-\mathcal{Q}(t)=2\left(c_{l}(t)-c_{l}^{0}\right)\mathcal{E}_{j}[\widetilde{R_{l}}](t)+2c_{l}^{0}\left(c_{l}(t)-c_{l}^{0}\right)\mathcal{M}_{j}[\widetilde{R_{l}}](t)+\left(c_{l}(t)-c_{l}^{0}\right)^{2}\mathcal{M}_{j}[\widetilde{R_{l}}](t),
\end{equation}
and so, because $\mathcal{M}_{j}$ and $\mathcal{E}_{j}$ are quadratic
in $\varepsilon$, we have that 
\begin{equation}
\left|\mathcal{Q}_{0}(t)-\mathcal{Q}(t)\right|\leq CA\left(a+e^{-\theta D}\right)\int\left(\varepsilon^{2}+\varepsilon_{x}^{2}\right)\Phi_{j},
\end{equation}
and so, if we take $a_{0}$ small enough and $D_{0}$ large enough,
we obtain 
\begin{equation}
\int\left(\varepsilon^{2}+\varepsilon_{x}^{2}+\varepsilon_{xx}^{2}\right)\Phi_{j}\leq C\mathcal{Q}(t).
\end{equation}
\end{proof}

\subsection{Proof of $\mathcal{P}_{j}$}

Now, we are left to prove $\mathcal{P}_{j}$. More precisely, if $P_{j}$
is a soliton, we will prove (\ref{2epsilon-ccl}) and (\ref{2velo-ccl});
if $P_{j}$ is a breather, we will prove (\ref{2epsilon-ccl}), (\ref{2eq:mass_var})
and (\ref{2eq:energ_var}). We will distinguish several cases. We assume
that $\sigma\leq\min\left(\zeta,\beta^{2}\right)$, $\theta\leq\min\left(\frac{\beta}{4},\frac{\sqrt{\sigma}}{16}\right)$,
and that $D_{0}$ is large enough and $a_{0}$ is small enough (depending
on $A,\theta$), so that all the previous lemmas are verified.

\subsubsection{Case when $P_{j}$ is a soliton}

\paragraph{\emph{Proof of (\ref{2epsilon-ccl})}}

By Lemma \ref{2lem:mono} and the fact that $b_j^2-a_j^2>0$ for $j\geq2$ (which is a direct consequence of $v_2>0$), we have that, if $j\geq2$, for any $t\in [0,t^*]$,
\begin{equation}
\mathcal{H}_{j}(t)-\mathcal{H}_{j}(0)\leq Ce^{-2\theta D}.\label{2monoto-gl}
\end{equation}
\eqref{2monoto-gl} is also true for $j=1$, because of Remark \ref{2rem:mono}.

We have, by (\ref{2quad-var}), (\ref{2monoto-gl}), the definition
of $\mathcal{Q}(0)$ and (\ref{2eq:first}) for $\varepsilon(0)$,
for any $t\in [0,t^*]$,
\begin{align}
\mathcal{Q}(t) & =\left[\mathcal{Q}(t)-\mathcal{Q}(0)-\left(\mathcal{H}_{j}(t)-\mathcal{H}_{j}(0)\right)\right]+\left[\mathcal{H}_{j}(t)-\mathcal{H}_{j}(0)\right]+\mathcal{Q}(0)\\
 & \leq Ce^{-2\theta D}+C\left(\frac{A}{Z_{j+1}}\right)^{2}\left(a^{2}+e^{-2\theta D}\right)+C\Vert\varepsilon(0)\Vert_{H^{2}}^{2}\\
 & \leq C\left(a^{2}+e^{-2\theta D}\right)+C\left(\frac{A}{Z_{j+1}}\right)^{2}\left(a^{2}+e^{-2\theta D}\right).\label{2Q(t)-maj}
\end{align}

And so, by (\ref{2coercivity}) and (\ref{2Q(t)-maj}), we have for
any $t\in [0,t^*]$,
\begin{align}
\int\left(\varepsilon^{2}+\varepsilon_{x}^{2}+\varepsilon_{xx}^{2}\right)\Phi_{j} & \leq C\mathcal{Q}(t)+C\left(\int\varepsilon\widetilde{P_{j}}\right)^{2}+Ce^{-2\theta D}\\
 & \leq C\left(a^{2}+e^{-2\theta D}\right)+C\left(\frac{A}{Z_{j+1}}\right)^{2}\left(a^{2}+e^{-2\theta D}\right)+C\left(\int\varepsilon\widetilde{P_{j}}\right)^{2}.\label{2maj-epsilon}
\end{align}

Because of (\ref{2orth-cond}) and $\widetilde{P_{j}}$ is a soliton,
we have that
\begin{equation}
\int\varepsilon\widetilde{P_{j}}=0.
\end{equation}

So, the proof of (\ref{2epsilon-ccl}) is completed for a suitable
constant $Z_{j}$ that will be precised later.
\paragraph{\emph{Proof of (\ref{2velo-ccl})}}

From (\ref{2mass-quadr}), (\ref{2solit-mass}), (\ref{2var-mass}),
(\ref{2varlinmass}), (\ref{2orth-cond}) and (\ref{2maj-epsilon}),
we have for any $t\in [0,t^*]$: 
\begin{align}
M_{j}(0)-M_{j}(t) & =M_{j}[\widetilde{P}](0)-M_{j}[\widetilde{P}](t)+\int\widetilde{P}\varepsilon\Phi_{j}(0)-\int\widetilde{P}\varepsilon\Phi_{j}(t)+\frac{1}{2}\int\varepsilon^{2}\Phi_{j}(0)-\frac{1}{2}\int\varepsilon^{2}\Phi_{j}(t)\label{2eq:mass-diff}\\
 & =\left(c_{l}(0)^{1/2}-c_{l}(t)^{1/2}\right)M[Q]+O\left(\left(a^{2}+e^{-2\theta D}\right)\right)+O\left(\left(\frac{A}{Z_{j+1}}\right)^{2}\left(a^{2}+e^{-2\theta D}\right)\right).\nonumber 
\end{align}

Similarly, from (\ref{2energ-quadr}), (\ref{2solit-energ}), (\ref{2var-energ}),
(\ref{2varlinenerg}), (\ref{2orth-energ}) and (\ref{2maj-epsilon}),
by taking $a_{0}$ smaller and $D_{0}$ larger with respect to $A$
if needed, we have for any $t\in [0,t^*]$: 
\begin{equation}
E_{j}(0)-E_{j}(t)=\left(c_{l}(0)^{3/2}-c_{l}(t)^{3/2}\right)E[Q]+O\left(\left(a^{2}+e^{-2\theta D}\right)\right)+O\left(\left(\frac{A}{Z_{j+1}}\right)^{2}\left(a^{2}+e^{-2\theta D}\right)\right).\label{2eq:energ-diff}
\end{equation}

For $\eta=\frac{1}{2},\frac{3}{2}$, because we know that $c_{l}(t)$
is not too far from $c_{l}(0)$ (and the both are not too far from
$c_{l}^{0}$, by (\ref{2eq:second})), we can write for any $t\in [0,t^*]$:
\begin{equation}
c_{l}(t)^{\eta}=\left(c_{l}(0)\right)^{\eta}\left(1+\eta\frac{c_{l}(t)-c_{l}(0)}{c_{l}(0)}+O\left(\left(c_{l}(t)-c_{l}(0)\right)^{2}\right)\right),
\end{equation}
and so for any $t\in [0,t^*]$,
\begin{equation}
c_{l}(t)^{\eta}-c_{l}(0)^{\eta}=\eta c_{l}(0)^{\eta-1}\left(c_{l}(t)-c_{l}(0)\right)+O\left(\left(c_{l}(t)-c_{l}(0)\right)^{2}\right),
\end{equation}
and if $a_{0}$ is small enough and $D_{0}$ is large enough, we have
by (\ref{2cond2}), for any $t\in [0,t^*]$:
\begin{equation}
\left|O\left(\left(c_{l}(t)-c_{l}(0)\right)^{2}\right)\right|\leq\frac{1}{2}\eta c_{l}(0)^{\eta-1}\left|c_{l}(t)-c_{l}(0)\right|.
\end{equation}

Thus, for any $t\in [0,t^*]$,
\begin{equation}
2\eta c_{l}(0)^{\eta-1}\left|c_{l}(t)-c_{l}(0)\right|\geq\left|c_{l}(t)^{\eta}-c_{l}(0)^{\eta}\right|\geq\frac{\eta c_{l}(0)^{\eta-1}}{2}\left|c_{l}(t)-c_{l}(0)\right|,\label{2eq:ineqs}
\end{equation}
where $c_{l}(0)^{\eta-1}$ is between $\min\{\frac{c_{p}^{0}}{2},1\leq p\leq L\}^{\eta-1}$
and $\max\left\{ 2c_{p}^{0},1\leq p\leq L\right\} ^{\eta-1}$ by (\ref{2eq:second}),
and so is bounded above and below by a constant that depends only
on the shape parameters of the solitons. In order to bound $\vert c_{l}(t)-c_{l}(0)\vert$
for a given $t\in [0,t^*]$, we will distinguish two cases.

\subparagraph{\underline{Case when $c_{l}(t)-c_{l}(0)\protect\geq0$}}

From (\ref{2eq:mass-diff}) and (\ref{2eq:ineqs}) for $\eta=\frac{1}{2}$,
we can say that 
\begin{align}
\left|c_{l}(t)-c_{l}(0)\right| & \leq C\left(c_{l}(t)^{1/2}-c_{l}(0)^{1/2}\right)\label{2eq:-4}\\
 & \leq C\frac{M_{j}(t)-M_{j}(0)}{M[Q]}+C\left(a^{2}+e^{-2\theta D}\right)+C\left(\frac{A}{Z_{j+1}}\right)^{2}\left(a^{2}+e^{-2\theta D}\right).\nonumber 
\end{align}

Now, $M[Q]>0$ and from Lemma \ref{2lem:mono},
\begin{equation}
\left|c_{l}(t)-c_{l}(0)\right|\leq C\left(a^{2}+e^{-2\theta D}\right)+C\left(\frac{A}{Z_{j+1}}\right)^{2}\left(a^{2}+e^{-2\theta D}\right).
\end{equation}

Thus, (\ref{2velo-ccl}) is established for a suitable constant $Z_{j}$
that will be precised later.

\subparagraph{\underline{Case when $c_{l}(t)-c_{l}(0)\protect\leq0$}}

From (\ref{2eq:energ-diff}) and (\ref{2eq:ineqs}) for $\eta=\frac{3}{2}$,
we can say that 
\begin{align}
\left|c_{l}(t)-c_{l}(0)\right| & \leq C\left(c_{l}(0)^{3/2}-c_{l}(t)^{3/2}\right)\label{2eq:-3}\\
 & \leq C\frac{E_{j}(0)-E_{j}(t)}{E[Q]}+C\left(a^{2}+e^{-2\theta D}\right)+C\left(\frac{A}{Z_{j+1}}\right)^{2}\left(a^{2}+e^{-2\theta D}\right).\nonumber 
\end{align}

Now, $E[Q]<0$ and from Lemma \ref{2lem:mono}, (\ref{2eq:mass-diff})
and (\ref{2eq:ineqs}) for $\eta=\frac{1}{2}$,
\begin{align}
\left|c_{l}(t)-c_{l}(0)\right| & \leq C\omega_{1}\left(M_{j}(0)-M_{j}(t)\right)+C\left(a^{2}+e^{-2\theta D}\right)+C\left(\frac{A}{Z_{j+1}}\right)^{2}\left(a^{2}+e^{-2\theta D}\right)\label{2eq:-2}\\
 & \leq C\omega_{1}\left(c_{l}(0)^{1/2}-c_{l}(t)^{1/2}\right)+C\left(a^{2}+e^{-2\theta D}\right)+C\left(\frac{A}{Z_{j+1}}\right)^{2}\left(a^{2}+e^{-2\theta D}\right)\nonumber \\
 & \leq C\omega_{1}\left|c_{l}(t)-c_{l}(0)\right|+C\left(a^{2}+e^{-2\theta D}\right)+C\left(\frac{A}{Z_{j+1}}\right)^{2}\left(a^{2}+e^{-2\theta D}\right),\nonumber 
\end{align}
and so, by taking $\omega_{1}$ small enough, we may deduce the desired
inequality: 
\begin{equation}
\left|c_{l}(t)-c_{l}(0)\right|\leq C\left(a^{2}+e^{-2\theta D}\right)+C\left(\frac{A}{Z_{j+1}}\right)^{2}\left(a^{2}+e^{-2\theta D}\right).
\end{equation}

Thus, (\ref{2velo-ccl}) is established for a suitable constant $Z_{j}$
that will be precised later.

\subsubsection{Case when $P_{j}$ is a breather}

\paragraph{\emph{Preliminaries}}

By the same argument as in the case when $P_{j}$ is a soliton,
we establish (\ref{2maj-epsilon}). However, we are not able to prove
(\ref{2epsilon-ccl}) immediately, because $\int\widetilde{P_{j}}\varepsilon$
is not necessarily equal to $0$ in the case when $\widetilde{P_{j}}$
is a breather.

From Lemma \ref{2lem:lin}, (\ref{2mass-quadr}), Lemma \ref{2lem:const}
and Remark \ref{2rem:const}, we have that for any $t\in [0,t^*]$,

\begin{align}
\int\widetilde{P_{j}}\varepsilon(t)-\int\widetilde{P_{j}}\varepsilon(0) & =M_{j}(t)-M_{j}(0)-\frac{1}{2}\int\varepsilon^{2}\Phi_{j}(t)+\frac{1}{2}\int\varepsilon^{2}\Phi_{j}(0)\nonumber\\
& +O\left(e^{-2\theta D}\right)+O\left(\left(\frac{A}{Z_{j+1}}\right)^{2}\left(a^{2}+e^{-2\theta D}\right)\right).
\end{align}

then we use (\ref{2mono-m}) and (\ref{2eq:first}) for $\varepsilon(0)$,
we have that for any $t\in [0,t^*]$,
\begin{align}
\int\widetilde{P_{j}}\varepsilon(t)-\int\widetilde{P_{j}}\varepsilon(0) & \leq-\frac{1}{2}\int\varepsilon^{2}\Phi_{j}(t)+\frac{1}{2}\int\varepsilon^{2}\Phi_{j}(0)+O\left(e^{-2\theta D}\right)+O\left(\left(\frac{A}{Z_{j+1}}\right)^{2}\left(a^{2}+e^{-2\theta D}\right)\right)\nonumber\\
 & \leq C\left(a^{2}+e^{-2\theta D}\right)+C\left(\frac{A}{Z_{j+1}}\right)^{2}\left(a^{2}+e^{-2\theta D}\right).\label{2eq:rhs} 
\end{align}

Now, from (\ref{2ellip-br}), Lemma \ref{2lem:lin}, (\ref{2mass-quadr}),
Lemma \ref{2lem:const}, Remark \ref{2rem:const} and (\ref{2cond1}),
we have for any $t\in [0,t^*]$,
\begin{align}
\left(\left(a_{j}^{2}+b_{j}^{2}\right)^{2}\right)\left(\int\widetilde{P_{j}}\varepsilon(0)-\int\widetilde{P_{j}}\varepsilon(t)\right) & =2\left(b_{j}^{2}-a_{j}^{2}\right)\left(e[\widetilde{P_{j}}](t)-e[\widetilde{P_{j}}](0)\right)+\left(f[\widetilde{P_{j}}](t)-f[\widetilde{P_{j}}](0)\right)\nonumber \\
 & =2\left(b_{j}^{2}-a_{j}^{2}\right)\left(E_{j}(t)-E_{j}(0)-\mathcal{E}_{j}(t)+\mathcal{E}_{j}(0)\right)\nonumber \\
 & +\left(F_{j}(t)-F_{j}(0)-\mathcal{F}_{j}(t)+\mathcal{F}_{j}(0)\right)\nonumber \\
 & +O\left(e^{-2\theta D}\right)+O\left(\left(\frac{A}{Z_{j+1}}\right)^{2}\left(a^{2}+e^{-2\theta D}\right)\right).\label{2eq:ell}
\end{align}

Recall that if $j\geq2$, $b_j^2-a_j^2>0$. This is why, from Lemma \ref{2lem:mono}
and (\ref{2eq:first}) for $\varepsilon(0)$, we have that, if $j\geq2$,
for any $t\in [0,t^*]$: 
\begin{align}
\int\widetilde{P_{j}}\varepsilon(0)-\int\widetilde{P_{j}}\varepsilon(t) & \leq\frac{-2\left(b_{j}^{2}-a_{j}^{2}\right)\mathcal{E}_{j}(t)-\mathcal{F}_{j}(t)}{\left(a_{j}^{2}+b_{j}^{2}\right)^{2}}\nonumber \\
 & +C\left(a^{2}+e^{-2\theta D}\right)+C\left(\frac{A}{Z_{j+1}}\right)^{2}\left(a^{2}+e^{-2\theta D}\right)\nonumber \\
 & +\frac{2\left(b_{j}^{2}-a_{j}^{2}\right)\omega_{1}+\omega_{2}}{\left(a_{j}^{2}+b_{j}^{2}\right)^{2}}\left(M_{j}(0)-M_{j}(t)\right).\label{2eq:ell2}
\end{align}
\eqref{2eq:ell2} is also true if $j=1$ because of Remark \ref{2rem:mono}.

And, from (\ref{2mass-quadr}), Lemma \ref{2lem:const}, Remark \ref{2rem:const},
Lemma \ref{2lem:lin} and (\ref{2eq:first}) for $\varepsilon(0)$,
we have for any $t\in [0,t^*]$, 
\begin{align}
M_{j}(0)-M_{j}(t) & =M_{j}[\widetilde{P}](0)-M_{j}[\widetilde{P}](t)+\int\widetilde{P}\varepsilon\Phi_{j}(0)-\int\widetilde{P}\varepsilon\Phi_{j}(t)+\frac{1}{2}\int\varepsilon^{2}\Phi_{j}(0)-\frac{1}{2}\int\varepsilon^{2}\Phi_{j}(t)\nonumber \\
 & =O\left(a^{2}+e^{-2\theta D}\right)+O\left(\left(\frac{A}{Z_{j+1}}\right)^{2}\left(a^{2}+e^{-2\theta D}\right)\right)+\int\widetilde{P_{j}}\varepsilon(0)-\int\widetilde{P_{j}}\varepsilon(t)-\mathcal{M}_{j}(t).\label{2mass1}
\end{align}

And so, if we choose $\omega_{1}$ and $\omega_{2}$ small enough
with respect to the problem constants, we obtain for any $t\in [0,t^*]$:
\begin{align}
\int\widetilde{P_{j}}\varepsilon(0)-\int\widetilde{P_{j}}\varepsilon(t) & \leq\frac{-\left(2\left(b_{j}^{2}-a_{j}^{2}\right)\omega_{1}+\omega_{2}\right)\mathcal{M}_{j}(t)-2\left(b_{j}^{2}-a_{j}^{2}\right)\mathcal{E}_{j}(t)-\mathcal{F}_{j}(t)}{\left(a_{j}^{2}+b_{j}^{2}\right)^{2}}\nonumber \\
 & +C\left(a^{2}+e^{-2\theta D}\right)+C\left(\frac{A}{Z_{j+1}}\right)^{2}\left(a^{2}+e^{-2\theta D}\right).\label{2eq:ell3}
\end{align}

Because $\left|\mathcal{M}_{j}(t)\right|+\left|\mathcal{E}_{j}(t)\right|+\left|\mathcal{F}_{j}(t)\right|\leq\int\left(\varepsilon^{2}+\varepsilon_{x}^{2}+\varepsilon_{xx}^{2}\right)\Phi_{j}$,
we deduce that for any $t\in [0,t^*]$,
\begin{equation}
\int\widetilde{P_{j}}\varepsilon(0)-\int\widetilde{P_{j}}\varepsilon(t)\leq C\int\left(\varepsilon^{2}+\varepsilon_{x}^{2}+\varepsilon_{xx}^{2}\right)\Phi_{j}+C\left(a^{2}+e^{-2\theta D}\right)+C\left(\frac{A}{Z_{j+1}}\right)^{2}\left(a^{2}+e^{-2\theta D}\right).\label{2eq:lhs}
\end{equation}

And so, by putting (\ref{2eq:rhs}) and (\ref{2eq:lhs}) together, we
have that for any $t\in [0,t^*]$: 
\begin{equation}
\left|\int\widetilde{P_{j}}\varepsilon(0)-\int\widetilde{P_{j}}\varepsilon(t)\right|\leq C\int\left(\varepsilon^{2}+\varepsilon_{x}^{2}+\varepsilon_{xx}^{2}\right)\Phi_{j}+C\left(a^{2}+e^{-2\theta D}\right)+C\left(\frac{A}{Z_{j+1}}\right)^{2}\left(a^{2}+e^{-2\theta D}\right).\label{2first-ccl-br}
\end{equation}

\paragraph{\emph{Proof of (\ref{2epsilon-ccl})}}

From (\ref{2eq:first}), (\ref{2first-ccl-br}), we deduce for any $t\in [0,t^*]$:
\begin{align}
\left|\int\widetilde{P_{j}}\varepsilon(t)\right| & \leq\left|\int\widetilde{P_{j}}\varepsilon(0)-\int\widetilde{P_{j}}\varepsilon(t)\right|+\left|\int\widetilde{P_{j}}\varepsilon(0)\right|\nonumber \\
 & \leq C\int\left(\varepsilon^{2}+\varepsilon_{x}^{2}+\varepsilon_{xx}^{2}\right)\Phi_{j}+C\left(a+e^{-2\theta D}\right)+C\left(\frac{A}{Z_{j+1}}\right)^{2}\left(a^{2}+e^{-2\theta D}\right)\nonumber \\
 & \leq CA^{2}\left(a^{2}+e^{-2\theta D}\right)+C\left(a+e^{-2\theta D}\right).\label{2eq:-1}
\end{align}

And so, if $a_{0}$ is small enough and $D_{0}$ is large enough,
for any $t\in [0,t^*]$,
\begin{equation}
\left(\int\widetilde{P_{j}}\varepsilon(t)\right)^{2}\leq C\left(\frac{A}{Z_{j+1}}\right)^{2}\left(a^{2}+e^{-2\theta D}\right)+C\left(a^{2}+e^{-2\theta D}\right).
\end{equation}

This is why, from (\ref{2maj-epsilon}), for any $t\in [0,t^*]$,
\begin{equation}
\int\left(\varepsilon^{2}+\varepsilon_{x}^{2}+\varepsilon_{xx}^{2}\right)\Phi_{j}\leq C\left(a^{2}+e^{-2\theta D}\right)+C\left(\frac{A}{Z_{j+1}}\right)^{2}\left(a^{2}+e^{-2\theta D}\right),\label{2epsilon-ccl-br}
\end{equation}
and (\ref{2epsilon-ccl}) is established for a suitable constant $Z_{j}$
that will be precised later.

\paragraph{\emph{Proof of (\ref{2eq:mass_var})}}

From (\ref{2epsilon-ccl-br}) and (\ref{2first-ccl-br}), for any $t\in [0,t^*]$:
\begin{equation}
\left|\int\widetilde{P_{j}}\varepsilon(0)-\int\widetilde{P_{j}}\varepsilon(t)\right|\leq C\left(\frac{A}{Z_{j+1}}\right)^{2}\left(a^{2}+e^{-2\theta D}\right)+C\left(a^{2}+e^{-2\theta D}\right).\label{2mass2}
\end{equation}

Thus, (\ref{2eq:mass_var}) is established for a suitable constant
$Z_{j}$ that will be precised later.

\paragraph{\emph{Proof of (\ref{2eq:energ_var})}}

From Lemma \ref{2lem:lin}, (\ref{2energ-quadr}), Lemma \ref{2lem:const},
Remark \ref{2rem:const}, (\ref{2first-ccl-br}) and (\ref{2epsilon-ccl-br}),
for any $t\in [0,t^*]$,
\begin{align}
e[\widetilde{P_{j}}](t)-e[\widetilde{P_{j}}](0) & =E_{j}(t)-E_{j}(0)-\mathcal{E}_{j}(t)+\mathcal{E}_{j}(0)\nonumber \\
 & +O\left(e^{-2\theta D}\right)+O\left(\left(\frac{A}{Z_{j+1}}\right)^{2}\left(a^{2}+e^{-2\theta D}\right)\right)\nonumber \\
 & =E_{j}(t)-E_{j}(0)+O\left(a^{2}+e^{-2\theta D}\right)+O\left(\left(\frac{A}{Z_{j+1}}\right)^{2}\left(a^{2}+e^{-2\theta D}\right)\right).\label{2eq:-28}
\end{align}

Now, from Lemma \ref{2lem:mono}, for any $t\in [0,t^*]$,
\begin{equation}
e[\widetilde{P_{j}}](t)-e[\widetilde{P_{j}}](0)\leq\omega_{1}\left(M_{j}(0)-M_{j}(t)\right)+O\left(a^{2}+e^{-2\theta D}\right)+O\left(\left(\frac{A}{Z_{j+1}}\right)^{2}\left(a^{2}+e^{-2\theta D}\right)\right).
\end{equation}

And from (\ref{2mass1}) and (\ref{2mass2}), for any $t\in [0,t^*]$,
\begin{equation}
e[\widetilde{P_{j}}](t)-e[\widetilde{P_{j}}](0)\leq C\left(\frac{A}{Z_{j+1}}\right)^{2}\left(a^{2}+e^{-2\theta D}\right)+C\left(a^{2}+e^{-2\theta D}\right).
\end{equation}

To bound above $e[\widetilde{P_{j}}](0)-e[\widetilde{P_{j}}](t)$, we
do as in (\ref{2eq:ell}), (\ref{2eq:ell2}), (\ref{2mass1}), (\ref{2eq:ell3})
and (\ref{2eq:lhs}), but with the energy instead of the mass, and
after using (\ref{2epsilon-ccl-br}), we obtain for any $t\in [0,t^*]$,
\begin{equation}
\left|e[\widetilde{P_{j}}](t)-e[\widetilde{P_{j}}](0)\right|\leq C\left(\frac{A}{Z_{j+1}}\right)^{2}\left(a^{2}+e^{-2\theta D}\right)+C\left(a^{2}+e^{-2\theta D}\right).
\end{equation}

Thus, (\ref{2eq:energ_var}) is established for a suitable constant
$Z_{j}$ that will be precised later.

\subsection{Choice of suitable $A$ and $Z_{j}$}

The induction holds if 
\begin{equation}
C\left(1+\left(\frac{A}{Z_{j+1}}\right)^{2}\right)\leq\left(\frac{A}{Z_{j}}\right)^{2}.
\end{equation}

We can set 
\begin{equation}
A:=\sqrt{2}\left(2C\right)^{\frac{J}{2}},
\end{equation}
and for $1\leq j\leq J$, 
\begin{equation}
Z_{j}:=2\left(2C\right)^{\frac{j-1}{2}}.
\end{equation}

And, if $C>1$, the induction holds.

The proof of Proposition \ref{2prop:ind} is now complete.

\section{A consequence of Theorem \ref{2:THM:MAIN}: orbital stability of a multi-breather}
\label{2sec:cons}

We assume the Theorem \ref{2:THM:MAIN} proved, let us prove Theorem
\ref{2thm:mb}.
\begin{proof}[Proof of Theorem \ref{2thm:mb}]
Let $A_{0}>0$, $\theta_{0}>0$, $D_{0}>0$ and $a_{0}>0$ from Theorem \ref{2:THM:MAIN}
(these constants do only depend on the shape/frequency parameters of our objects
and not on their initial positions). Let $\eta_0>\eta>0$ with $\eta_0<\frac{a_0}{2C_1}$ and $C_1$ defined in the following. Let $a<a_{0}$ and $D>D_{0}$
such that $A_{0}\left(a+e^{-\theta_{0}D}\right)<4A_0C_1\eta$ and $a=2C_1 \eta$.
We may take $D$ even larger so that $T^{*}=0$ where $T^{*}$ is
defined in \cite[Theorem 1.2]{key-49}. Let $\theta_1$ associated to shape/frequency parameters $\alpha_k,\beta_k,\kappa_l,c_0^l$ by \cite[Theorem 1.2]{key-49}.
Let $A_2$ associated to $D$ by \cite[Theorem 1.2]{key-49}.

Let $\tau>0$ be the minimal difference between two velocities.

Let
$p$ be the multi-breather associated to $\alpha_{k},\beta_{k},x_{1,k}^{0},x_{2,k}^{0},c_{l}^{0},x_{0,l}^{0}$
by \cite[Theorem 1.2]{key-49} with notations as in (\ref{2eq:mb})
and $P$ the corresponding sum with notations as in (\ref{2eq:sosb}).
We may choose $T\ge0$ large enough such that
\begin{equation}
\forall t\geq T,\quad\left\Vert p(t)-P(t)\right\Vert _{H^{2}}\leq a/2,\label{eq:kpv}
\end{equation}
which is possible from \cite{key-49}, and such that 
\begin{equation}
\forall j\geq2,\quad x_{j}(T)-x_{j-1}(T)>2D,
\end{equation}
which is possible because the distance between two objects is increasing
with a speed that is at least $\tau$.

By \cite{key-19} we know that we have continuous dependence of the
solution of (mKdV) with respect to the initial data. And so, there
exists $C_1>0$ (that depends on $T$) such that if $\Vert u(0)-p(0)\Vert_{H^{2}}\leq\eta$ and $\eta_0$ is small enough,
then 
\begin{equation}
\forall t\in[0,T],\quad\Vert u(t)-p(t)\Vert_{H^{2}}\leq C_1 \eta.\label{eq:kpv2}
\end{equation}

Therefore, by triangular inequality, 
\begin{equation}
\left\Vert u(T)-P(T)\right\Vert _{H^{2}}\leq a=2C_1\eta.
\end{equation}

This means that the assumptions of Theorem \ref{2:THM:MAIN} are all
satisfied in $T$ instead of $0$. And so, this means that there exists $x_{0,l}(t),x_{1,k}(t),x_{2,k}(t)$
defined for any $t\geq T$ such that 
\begin{equation}
\forall t\geq T,\quad\left\Vert u(t)-P(t;\alpha_{k},\beta_{k},x_{1,k}(t),x_{2,k}(t),\kappa_{l},c_{l}^{0},x_{0,l}(t))\right\Vert _{H^{2}}\leq A_0(a+e^{-\theta_0D})<4A_0C_1\eta.
\end{equation}

Now, we see that the assumptions of \cite[Theorem 1.2, Remark 1.3]{key-49}
are all satisfied in any $t\geq T$ instead of $0$ for the sum $P(\alpha_{k},\beta_{k},x_{1,k}(t),x_{2,k}(t),\kappa_{l},c_{l}^{0},x_{0,l}(t))$.
Indeed, if we denote $\widetilde{x_{j}}(t)$ the position of 
\begin{equation}
P_{j}(t;\alpha_{k},\beta_{k},x_{1,k}(t),x_{2,k}(t),\kappa_{l},c_{l}^{0},x_{0,l}(t)),
\end{equation}
we know from Remark \ref{2cons5} that for any $t\geq T$ and
$j\geq2$,
\begin{equation}
\widetilde{x_{j}}(t)-\widetilde{x_{j-1}}(t)\geq D.
\end{equation}

By taking $T$ larger if needed, we may insure that $e^{-\theta_1 T} < C_1\eta$.

Therefore, for any $t\geq T$, 
\begin{align*}
\left\Vert p(t;\alpha_{k},\beta_{k},x_{1,k}(t),x_{2,k}(t),\kappa_{l},c_{l}^{0},x_{0,l}(t))-P(t;\alpha_{k},\beta_{k},x_{1,k}(t),x_{2,k}(t),\kappa_{l},c_{l}^{0},x_{0,l}(t))\right\Vert _{H^{2}}\\
\leq A_{2}e^{-\theta_1 t}\leq A_{2}e^{-\theta_1 T}\leq A_2C_1\eta.
\end{align*}

And so, by triangular inequality, 
\begin{equation}
\forall t\geq T,\quad\left\Vert u(t)-p(t;\alpha_{k},\beta_{k},x_{1,k}(t),x_{2,k}(t),\kappa_{l},c_{l}^{0},x_{0,l}(t))\right\Vert _{H^{2}}\leq (4A_0+A_2)C_1\eta.
\end{equation}

The latter proves the Theorem for $t\geq T$ with $C_0=(4A_0+A_2)C_1$. For $0\leq t\leq T$, it is
enough to use (\ref{eq:kpv2}). 
\end{proof}

\section*{Appendix: Equations for localized conservation laws}
\begin{lem}
Let $f$ be a $C^{3}$ function that do not depend on time and $u$
a solution of \eqref{2mKdV}. Then, 
\begin{equation}
\frac{d}{dt}\frac{1}{2}\int u^{2}f=\int\left[-\frac{3}{2}u_{x}^{2}+\frac{3}{4}u^{4}\right]f_{x}+\frac{1}{2}\int u^{2}f_{xxx},
\end{equation}
\begin{equation}
\frac{d}{dt}\int\left[\frac{1}{2}u_{x}^{2}-\frac{1}{4}u^{4}\right]f=\int\left[-\frac{1}{2}\left(u_{xx}+u^{3}\right)^{2}-u_{xx}^{2}+3u^{2}u_{x}^{2}\right]f_{x}+\frac{1}{2}\int u_{x}^{2}f_{xxx},
\end{equation}
\begin{align}
\MoveEqLeft
\frac{d}{dt}\int\left[\frac{1}{2}u_{xx}^{2}-\frac{5}{2}u^{2}u_{x}^{2}+\frac{1}{4}u^{6}\right]f\nonumber \\
 & =\int\left(-\frac{3}{2}u_{xxx}^{2}+9u_{xx}^{2}u^{2}+15u_{x}^{2}uu_{xx}+\frac{9}{16}u^{8}+\frac{1}{4}u_{x}^{4}+\frac{3}{2}u_{xx}u^{5}-\frac{45}{4}u^{4}u_{x}^{2}\right)f'\nonumber \\
 & +5\int u^{2}u_{x}u_{xx}f''+\frac{1}{2}\int u_{xx}^{2}f'''.\label{2eq:-24}
\end{align}
\end{lem}

\begin{proof}[Proof]
see the bottom of the page 1115 and the bottom of the page 1116 of
\cite{key-2} and Section 5.5 in \cite{key-49}. 
\end{proof}
\begin{lem}
Let $r>0$. If $t,x$ satisfy $\widetilde{x_{j-1}}(t)+r<x<\widetilde{x_{j}}(t)-r$,
then 
\begin{equation}
\left|\widetilde{P}(t,x)\right|\leq Ce^{-\beta r}.
\end{equation}
\end{lem}

\begin{proof}[Proof]
immediate consequence of the exponential majoration of each object. 
\end{proof}


\bibliography{bib}
\bibliographystyle{siam}

\begin{center}
IRMA, UMR 7501, Université de Strasbourg, CNRS, F-67000 Strasbourg, France

semenov@math.unistra.fr

\end{center}

\end{document}